\documentclass[aihp]{imsart}

\RequirePackage{amsthm,amsmath,amsfonts,amssymb}
\RequirePackage[numbers,sort&compress]{natbib}
\RequirePackage[colorlinks,citecolor=blue,urlcolor=blue]{hyperref}
\RequirePackage{graphicx}

\usepackage{amsmath, amssymb, amsthm, bbm}
\usepackage{units}
\usepackage{fancybox}
\usepackage{hyperref}
\usepackage{enumerate} 

\usepackage{dsfont}  

\usepackage{csquotes}

\startlocaldefs

\theoremstyle{plain}
\newtheorem{theorem}{Theorem}[section]
\newtheorem{prop}[theorem]{Proposition}
\newtheorem{setting}[theorem]{Setting}
\newtheorem{lemma}[theorem]{Lemma}

\theoremstyle{remark}
\newtheorem{definition}[theorem]{Definition}

\providecommand{\1}{{ \ensuremath{ \mathbbm{1}}}}
\providecommand{\E}{{\ensuremath{\mathbb{E}}}}

\newcommand{ \N}{ \mathbb{N}}
\newcommand{ \R}{ \mathbb{R}}

\newcommand{\mycdot}{\ensuremath\cdot}
\newcommand{\symbolsigmaalgebra}{\mathfrak{S}}
\endlocaldefs

\begin{document}

\begin{frontmatter}
\title{On the It\^o-Alekseev-Gr\"obner formula\\
for stochastic differential equations}
\runtitle{It\^o-Alekseev-Gr\"obner formula}

\begin{aug}
\author[A]{\inits{A.}\fnms{Anselm}~\snm{Hudde}\ead[label=e1]{hudde@rheinahrcampus.de}\orcid{0000-0002-5652-2815}},
\author[B]{\inits{M.}\fnms{Martin}~\snm{Hutzenthaler}\ead[label=e2]{martin.hutzenthaler@uni-due.de}\orcid{0000-0003-0738-8717}}
\author[C,D]{\inits{A.}\fnms{Arnulf}~\snm{Jentzen}\ead[label=e3]{ajentzen@cuhk.edu.cn}\ead[label=e4]{ajentzen@uni-muenster.de}\orcid{0000-0002-9840-3339}}
\and
\author[E]{\inits{S.}\fnms{Sara}~\snm{Mazzonetto}\ead[label=e5]{sara.mazzonetto@univ-lorraine.fr}\orcid{0000-0001-6187-2716}}

\address[A]{Faculty of Mathematics and Technology, University of Applied Sciences Koblenz, Germany\printead[presep={,\ }]{e1}}

\address[B]{Faculty of Mathematics, University of Duisburg-Essen, Essen, Germany\printead[presep={,\ }]{e2}}
\address[C]{School of Data Science and Shenzhen Research Institute of Big Data,
The Chinese\\ University of Hong Kong, Shenzen (CUHK-Shenzen), China\printead[presep={,\ }]{e3}}
\address[D]{Applied Mathematics: Institute for Analysis and Numerics, Faculty of Mathematics\\ and Computer Science, University of M\"unster, Germany\printead[presep={,\ }]{e4}}
\address[E]{Facult\'e des sciences et Technologies, Universit\'e de Lorraine, France\printead[presep={,\ }]{e5}}
\end{aug}

\begin{abstract}
In this article we establish a new formula for the difference of
a test function of
the solution of a stochastic differential equation and of the test function
of an It\^o process.
The introduced formula essentially generalizes both 
the classical Alekseev-Gr\"obner formula from  the literature on deterministic
differential equations as well as the classical It\^o formula from stochastic analysis.
The discovered formula,
which we suggest to refer to as 
It\^o-Alekseev-Gr\"obner formula,
is a powerful tool for deriving
strong approximation rates for
perturbations and approximations
of stochastic ordinary and partial differential equations.

\end{abstract}

\begin{abstract}[language=french]
Dans cet article nous pr\'esentons une nouvelle formule qui exprime la diff\'erence entre une fonction test appliqu\'ee \`a une solution d’une \'equation diff\'erentielle stochastique et la m\^eme fonction test appliqu\'ee \`a un processus d’It\^o. Cette formule g\'en\'eralise \`a la foi la formule classique de Alekseev-Gr\"obner pour les \'equations diff\'erentielles d\'eterministes et la formule d’It\^o de l’analyse stochastique. Ainsi, nous sugg\'erons de l’appeler formule d’It\^o-Alekseev-Gr\"obner. Il s’agit d’un outil puissant pour d\'eriver les taux d’approximation forte pour des perturbations et approximations d’\'equations diff\'erentielles stochastiques et d’\'equations \`a d\'eriv\'ees partielles stochastiques.
\end{abstract}

\begin{keyword}[class=MSC]
\kwd{60H10}
\end{keyword}

\begin{keyword}
\kwd{It\^o formula}
\kwd{Alekseev-Gr\"obner formula}
\kwd{nonlinear variation-of-constants formula}
\kwd{nonlinear integration-by-parts formula}
\kwd{perturbation of stochastic differential equations}
\kwd{strong convergence rate}
\kwd{non-globally monotone coefficients}
\kwd{small-noise analysis}
\end{keyword}

\end{frontmatter}

\allowdisplaybreaks
\section{Introduction}

The linear integration-by-parts formula states in the simplest case that for
all $a,b\in\R$, $t\in[0,\infty)$ it holds that
\begin{equation} \begin{split} 
e^{at} -e^{bt} =-\int_0^t \tfrac{d}{ds} \big(e^{a(t-s)} e^{bs} \big)\,ds
=\int_0^t e^{a(t-s)} (a-b)e^{bs} \,ds.
\end{split} \end{equation} 
The nonlinear integration-by-parts formula, which is also referred to as \emph{Alekseev-Gr\"obner formula}
or as nonlinear variation-of-constants formula,
generalizes this relation
to nonlinear ordinary differential equations and
has been
established in Alekseev~\cite{Alekseev61} and Gr\"obner~\cite{Groebner60}.
More formally,
the Alekseev-Gr\"obner formula
(cf., e.g.,\ Hairer et al.~\cite[Theorem I.14.5]{HairerNorsettWannerI})
asserts
that for all $d\in\N$, $T\in(0,\infty)$,
$\mu\in C^{0,1}([0,T]\times\R^d,\R^d)$,
$Y\in C^1([0,T],\R^d)$,
and all $X_{\mycdot,\mycdot}^{\mycdot}=(X_{s,t}^x)_{s\in[0,t],t\in[0,T],x\in\R^d}
\in C(\{(s,t)\in[0,T]^2\colon s\leq t\}\times\R^d,\R^d)$
with $\forall\,s\in[0,T]$, $t\in[s,T]$, $x\in\R^d$:
$X_{s,t}^x=x+\int_s^t\mu(r,X_{s,r}^x)\,dr$
it holds that
\begin{equation} \begin{split} \label{eq:AGL} 
  X_{0,T}^{Y_0}-Y_T=\int_0^T \Big(\tfrac{\partial}{\partial x}X_{r,T}^{Y_r}\Big)\Big(\mu(r,Y_r)- \tfrac{d}{dr} Y_r\Big)\,dr.
\end{split} \end{equation} 
Informally speaking, the Alekseev-Gr\"obner formula expresses
the global error (the term $X_{0,T}^{Y_0}-Y_T$ in~\eqref{eq:AGL})
in terms of the infinitesimal error (the term $\mu(r,Y_r)- \tfrac{d}{dr} Y_r$
in~\eqref{eq:AGL} which corresponds to the difference of time derivatives).
For this reason, the
Alekseev-Gr\"obner formula is a powerful tool
for studying perturbations of ordinary differential equations;
see, e.g., Norsett \& Wanner~\cite[Theorem 3]{NorsettWanner1979},
Lie \& Norsett~\cite[Theorem 1]{LieNorsett1989}, Iserles \& Soederlind~\cite[Theorem 1]{IserlesSoederlind1993},
and Iserles~\cite[Theorem 3.7]{Iserles2009}.

In this article we generalize the Alekseev-Gr\"obner formula to a stochastic setting
and derive the nonlinear integration-by-parts formula
for stochastic differential equations (SDEs).
Informally speaking,
one key difficulty 
in this generalization is that the integrand on the right-hand side
of~\eqref{eq:AGL} (and a similar integrand
appears in the stochastic integral in~\eqref{1553} below)
depends both on the past
(e.g.\ the term $\mu(r,Y_r)$) and on the future
(e.g.\ the term $\tfrac{\partial}{\partial x}X_{r,T}^{Y_r} $).
This precludes a generalization which is solely based on It\^o calculus.
In this article we apply Malliavin calculus
and express anticipating stochastic integrals as Skorohod integrals.
The following theorem, Theorem~\ref{thm:intro}, formulates our main contribution and establishes -- what we call -- the It\^o-Alekseev-Gr\"obner formula. 
For its formulation and throughout this article we use the notation introduced in Subsection \ref{sec:notation} below.
\begin{theorem}[It\^o-Alekseev-Gr\"obner formula] \label{thm:intro} 
Let 
$d, m, k \in \N$, 
$T, c \in (0, \infty)$, 
$p \in(4, \infty)$, 
$q \in[0,\tfrac{p}{2}-2)$,
$\xi\in\R^d$,
let 
$( \Omega, \mathcal{F}, \mathbb P)$
be a probability space, 
let
$
W \colon [0, T] \times \Omega
\rightarrow \R^m
$
be a standard Brownian motion with continuous sample paths,
let $ \mathcal N = \{ A\in\mathcal{F}\colon \mathbb{P}(A)=0\}$,
let
$ \mu \in C\big([0, T] \times \R^d , \R^d\big)$,
$ \sigma \in C\big([0, T] \times \R^d , \R^{d\times m}\big)$,
let
$X_{ \mycdot, \mycdot}^{\mycdot} =(X_{ s, t}^x)_{s\in[0,t],t\in[0,T],x\in\R^d} \colon \{(s,t)\in[0,T]^2\colon s\leq t\} \times \R^d \times \Omega \to \R^d$
be a continuous random field,
assume that for all $s\in[0,T]$, $ \omega \in \Omega$ 
it holds that
$( \R^d\ni x \mapsto X_{s, T}^x( \omega) \in \R^d)\in C^{2}( \R^d, \R^d)$, 
assume that for all $\omega\in\Omega$ it holds that
$\tfrac{\partial^2}{\partial x^2}X_{\mycdot,T}^{\mycdot}(\omega)\in C( [0,T]\times\R^d, L^{(2)}(\R^d,\R^d))$,
assume that for all 
$ s \in [0, T]$, $x\in \R^d$
the stochastic process
$[s, T] \times \Omega \ni (t, \omega) \mapsto X_{s, t}^x(\omega) \in \R^d$
is
$( \symbolsigmaalgebra(\mathcal{N}\cup \symbolsigmaalgebra(W_r - W_s \colon r \in [s, t])))_{t \in [s, T]}$-adapted,
assume that for all
$ s \in [0, T]$, $t\in[s,T]$, $x\in \R^d$
it holds $ \mathbb P$-a.s.\ that 
\begin{equation} \label{eq:SDE.intro}
\begin{split} 
X_{s, t}^x
&=
x + \smallint_s^t \mu(r, X_{s, r}^x)\,dr + \smallint_s^t \sigma(r, X_{s, r}^x)\,dW_r,
\end{split}
\end{equation}
assume that for all $s\in[0,T]$, $t\in[s,T]$, $x \in \R^d$ it holds $ \mathbb P$-a.s.\ that 
$X_{t, T}^{X_{s, t}^x} = X_{s, T}^x$,
let
$A, Y \colon[0,T] \times\Omega \to \R^d$,
$B \colon[0,T] \times\Omega \to \R^{d\times m}$
be
$( \symbolsigmaalgebra(\mathcal{N}\cup\symbolsigmaalgebra(W_r \colon r \in [0, t])))_{t \in [0, T]} $-predictable
stochastic processes,
assume that $Y$ has continuous sample paths, 
assume that $\int_0^T \mathbb E \Big[\|A_s\|_{\R^d} ^p+\|Y_s\|_{\R^d} ^p+\|B_s\|_{L(\R^m,\R^d)} ^p\Big]\,ds<\infty$,
assume that for all $t \in [0, T]$ it holds $ \mathbb P$-a.s.\ that 
\begin{equation} 
Y_t
=
\xi + \smallint_0^t A_s \,ds + \smallint_0^t B_s \,dW_s, 
\end{equation} 
assume that
\begin{equation} 
\sup_{ \substack{ s, t \in [0, T] \\ s \leq t} } 
\mathbb E \bigg [
\Big \| \mu\Big(t, X_{ s, t} ^{Y_{ s } } \Big) \Big \|_{ \R^d } ^p 
+
\Big \| \sigma \Big(t, X_{ s, t} ^{Y_{ s } } \Big) \Big\|_{L(\R^m,\R^d)}^p
\bigg]
< \infty,
\end{equation} 
assume that
\begin{equation} 
\sup_{ \substack{r, s, t \in [0, T] \\ r \leq s \leq t} } 
\mathbb E \bigg[ \Big \| X_{t, T} ^{X_{r, s} ^{Y_r} } \Big \|_{ \R^d} ^p
+
\Big \|  \tfrac{ \partial }{ \partial x} X_{t, T}^{X_{r, s} ^{Y_r}} 
\Big \|_{L(\R^d, \R^d)} ^{  \frac{4p}{p - 2(q+2)} } 
+
\Big \|
 \tfrac{ \partial^2 }{ \partial x^2} X_{t, T}^{X_{r, s} ^{Y_r}}
\Big \|_{ L^{(2)} ( \R^d, \R^d)} ^{  \frac{2p}{p - 2(q+2)} } \bigg]
< 
\infty,
\end{equation} 	
and let
$f \in C^2(\R^d, \R^k)$ 
satisfy that for all $x \in \R^d$ 
it holds that 
\begin{equation} \label{my442}
\max \Big \{
\tfrac{ \| f(x) \|_{ \R^k}}{1 + \| x \|_{ \R^d}},
\| f'(x) \|_{L( \R^d, \R^k)}, 
\|f''(x) \|_{L^{(2)}( \R^d, \R^k)}
\Big \}
\leq
c (1 + \| x \|_{ \R^d }^q).
\end{equation}
Then
the stochastic process
$
\big(
f' \big(X_{r, T}^{Y_r} \big) \tfrac{ \partial}{ \partial x} X_{r, T}^{Y_r} 
( \sigma(r, Y_r) - B_r)
\big)_{r \in [0, T]} 
$
is Skorohod-integrable and 
it holds $ \mathbb P$-a.s.\ that 
\begin{equation} \label{1553} 
\begin{split} 
&f \big(X_{0, T}^{Y_0} \big)
- 
f (Y_T )
=
\smallint_0^T
f' \big(X_{r, T}^{Y_r} \big)
 \tfrac{ \partial}{ \partial x} X_{r, T}^{Y_r}
\Big( \mu(r, Y_r) - A_r \Big) 
\,dr
+
\smallint_0^T
f' \big(X_{r, T}^{Y_r} \big)
 \tfrac{ \partial}{ \partial x} X_{r, T}^{Y_r}
 \Big( \sigma(r, Y_r) - B_r\Big)
\,\delta W_r
\\&
+ 
\tfrac{1}{2} 
\sum_{i, j=1}^d
\smallint_0^T
\Big( \sigma(r, Y_r) [\sigma(r, Y_r)]^*
- 
B_r [B_r]^*
\Big)_{i, j}
\Big(
f'' \big(X_{r, T}^{Y_r} \big)
\big(\tfrac{\partial}{\partial x}X_{r, T}^{Y_r}, \tfrac{\partial}{\partial x}X_{r, T}^{Y_r} \big)
+ 
f' \big(X_{r, T}^{Y_r} \big)
\tfrac{\partial^2}{\partial x^2}X_{r, T}^{Y_r}
\Big)
\big( e_i^{(d)}, e_j^{(d)} \big)
\,dr.
\end{split} 
\end{equation} 
\end{theorem} 
\noindent
Theorem~\ref{thm:intro} follows immediately from Theorem~\ref{thm:perturbation.formula}
(applied with $\mathbb{F}_0=\symbolsigmaalgebra(\mathcal{N})$, $O = \R^d$
in the notation of Theorem~\ref{thm:perturbation.formula}).
Theorem~\ref{thm:intro} essentially generalizes the following results from the literature:
\begin{enumerate}[(i)]
  \item Theorem~\ref{thm:intro} essentially generalizes the Alekseev-Gr\"obner formula.
     More formally, Theorem~\ref{thm:intro} (applied with $\sigma=0$, $B=0$, $k=d$, $=\textup{Id}_{\R^d}$ in the notation of Theorem~\ref{thm:intro})
     implies 
     the Alekseev-Gr\"obner formula in~\eqref{eq:AGL}
     (cf., e.g.,\ Hairer et al.~\cite[Theorem I.14.5]{HairerNorsettWannerI})
     in the case where the solution process
     is twice continuously differentiable in the space variable.
  \item Theorem~\ref{thm:intro} essentially generalizes the It\^o formula.
     More formally, Theorem~\ref{thm:intro} (applied with $\mu=0$, $\sigma=0$ in the notation of Theorem~\ref{thm:intro})
     implies the It\^o formula for It\^o processes
     (cf., e.g., Revuz \& Yor~\cite[Theorem IV.3.3]{RevuzYor1994})
     in the case
     where the It\^o process
     $Y$, its drift process $A$, and its diffusion process $B$ satisfy
     $\inf_{p\in(4,\infty)}\big(\sup_{s\in[0,T]}\E\big[\|Y_s\|_{\R^d}^p\big]
     +\int_0^T\E\big[\|A_s\|_{\R^d}^p+\|B_s\|_{L(\R^m,\R^d)}^p\big]\,ds\big)<\infty$.
     This moment requirement is due to the fact that we use the Skorohod integral.
     An approach with rough path integrals (cf., e.g., Friz \& Hairer~\cite{FrizHairer2014})
     might be suitable to generalize Theorem~\ref{thm:intro}
     so that this moment condition would not be needed.
   \item Theorem~\ref{thm:intro} essentially generalizes 
     the Alekseev-Gr\"obner formula in~\eqref{eq:AGL}
     (cf., e.g.,\ Hairer et al.~\cite[Theorem I.14.5]{HairerNorsettWannerI})
     even in the deterministic case ($\sigma=0$ and $B=0$ in the notation of Theorem~\ref{thm:intro})
     from $f=\textup{Id}_{\R^d}$ to general test functions.
     In Proposition~\ref{p:IAG} below we prove the It\^o-Alekseev-Gr\"obner formula in~\eqref{1553} in the deterministic case
     with the test function $f\colon\R^d\to\R^k$ being only in $C^1(\R^d,\R^k)$
     instead of in $C^2(\R^d,\R^k)$ as in Theorem~\ref{thm:intro} above.
     The proof of 
     Proposition~\ref{p:IAG} below is also illustrative to understand the structure of the It\^o-Alekseev-Gr\"obner formula in~\eqref{1553}.
 \item Theorem~\ref{thm:intro} essentially provides a pathwise version of the well-known weak error expansion
     (cf., e.g., Graham \& Talay \cite[(7.48) and the last Display on page 182]{GrahamTalay2013}
     or related weak error estimates in~\cite{TalayTubaro1990,DebrabantRoessler2008,TambueNgnotchouye2016}).
     More precisely, in the notation of Theorem~\ref{thm:intro} taking expectation of~\eqref{1553},
     using that the expectation of the Skorohod integral vanishes, 
     and exchanging expectations and temporal
     integrals
     results in
     the standard
     representation of the weak error
     $\E\big[ f \big(X_{0, T}^{Y_0} \big)\big] - \E\big[f (Y_T )\big]$.
\end{enumerate}
We note that
\eqref{1553} roughly follows from the case $k=d$, $f=\textup{Id}_{\R^d}$
of Theorem~\ref{thm:intro} and from the It\^o formula for anticipating processes
established in Al\`os \& Nualart \cite{AlosNualart1998} applied to the anticipating
process $[0,T]\ni t\mapsto X_{t,T}^{Y_t}\in\R^d$.
Moreover, using the two-sided stochastic integral of
Pardoux \& Protter
\cite{PardouxProtter1987},
Nualart \& Pardoux \cite[Proposition 8.2]{NualartPardoux1988}
establish a backward It\^o-Ventzell-type formula
where the random test function roughly speaking 
has the form $F(t,x,\omega)=f(t,x,Y_t(\omega))$ where $f$ is deterministic
and $Y$ is a semimartingale.
This result is not applicable to the situation of Theorem~\ref{thm:intro} 
since $x\mapsto X^x$ is a random function.
However, our proof of Theorem~\ref{thm:intro} uses ideas
of the proof of
Nualart \& Pardoux \cite[Proposition 8.2]{NualartPardoux1988}.
Moreover, after the preprint \cite{HuddeHutzenthalerJentzenMazzonetto2018}
of our paper
appeared,
Del Moral \& Singh
\cite{delMoralSingh2019,delMoralSingh2022} establish
a backward It\^o-Ventzell formula
and use this to provide a new proof of Theorem~\ref{thm:intro}
in the special case of
coefficient functions which have continuous and uniformly bounded 
spatial derivatives
up to third order.
In addition, independently of our results 
(cf.\ \cite{ArnaudonDelMoral2018}
with our preprint
\cite{HuddeHutzenthalerJentzenMazzonetto2018}),
Arnaudon \& Del Moral
\cite[(3.2)]{ArnaudonDelMoral2019} arrive at
the It\^o-Alekseev-Gr\"obner
formula in a specific situation (e.g.\ the diffusion terms are equal)
by heuristically applying
the backward It\^o-Ventzell formula which was later
established in \cite{delMoralSingh2022}.

Theorem~\ref{thm:intro} implies immediately an $L^2$-estimate.
For example 
the $L^2$-norm of the right-hand side of~\eqref{1553}
can be bounded by the triangle inequality.
The $L^2$-norm of the Skorohod integral on the right-hand side of~\eqref{1553}
can then be calculated by applying the It\^o isometry for Skorohod integrals
(see, e.g., Alos \& Nualart~\cite[Lemma 4]{AlosNualart1998}).
Another approach for obtaining $L^2$-estimates is to
apply the It\^o formula for Skorohod processes
to the squared norm of the right-hand side of~\eqref{1553}.
However this seems to require additional regularity.

Our main motivation for the It\^o-Alekseev-Gr\"obner formula
are \textit{strong convergence rates} for \textit{time-discrete numerical approximations
of  stochastic evolution equations (SEEs)}. In the literature, positive strong convergence rates
have been established for SEEs with monotone nonlinearities (see, e.g.,
\cite[Chapter 4]{LiuRoeckner2015});
see, e.g., \cite{GyoengyMillet2007,KovacsLarssonLindgren2015,JentzenPusnik2020,BeckerGessJentzenKloeden2023,BrehierGoudenege2019,BrehierCuiHong2019,LiuQiao2020AllenCahnAdditive,Wang2020,BeckerJentzen2019} for the case of additive noise
and
\cite{
HutzenthalerJentzenKloeden2012,HutzenthalerJentzen2020,Sabanis2016,DareiotisKumarSabanis2016,NgoTaguchi2016,
Prohl2018,LiuQiao2021MultiplicativeNoise,
LeonhardRoessler2018,KellyLord2018} for the case of multiplicative noise;
for lower bounds see, e.g., \cite{dg01,mr07a,mrw08a,mrw08b,BeckerGessJentzenKloeden2020}.
Recently, the classical Alekseev-Gr\"obner formula has been applied
in \cite{HutzenthalerJentzenLindnerPusnik2019} to establish strong
convergence rates for space-time discrete approximations for stochastic Burgers
equations with \textit{additive noise} by rewriting the SEE as random partial differential equation.
This demonstrates that the Alekseev-Gr\"obner formula
is a successful approach for proving convergence rates in the case of SEEs
such as the stochastic Burgers equation with additive noise.
Now the It\^o-Alekseev-Gr\"obner formula in Theorem \ref{thm:intro}
provides an approach to derive strong convergence rates
e.g.\ for stochastic Burgers equations also in the case of \textit{non-additive noise}.
Applications of this approach are left to future research.

In addition, Theorem~\ref{thm:intro} can be applied to any approximation of an SDE which is an It\^o process
with respect to the same Wiener process driving the SDE.
Possible applications (cf., e.g., \cite{HutzenthalerJentzen2020})
include, in the notation of Theorem~\ref{thm:intro},
\begin{enumerate}[(i)]
\item
strong convergence rates for \emph{time-discrete numerical approximations
of SDEs} 
(e.g., the Euler-Maruyama approximation with $N\in\N$ time discretization steps
is given by
$ A_t = \mu( \tfrac{ k T }{ N },  Y_{ \frac{ k T }{ N } } ) $ 
and
$ B_t = \sigma( \tfrac{ k T }{ N }, Y_{ \frac{ k T }{ N } } ) $ 
for all $ t \in [ \frac{ kT }{ N }, \frac{ (k+1) T }{ N } )$, $ k \in \N_0\cap[0,N) $),
\label{i:problem2} 
\item
strong convergence rates for
\emph{Galerkin approximations
for SEEs} (see, e.g., \cite{Cerrai2001}) 
(choose $ A_t = P( \mu( t, Y_t ) ) $ 
and $ B_t u = P( \sigma( t, Y_t) u ) $ 
for all $ u \in \R^m $, $ t \in [0,T] $
and some suitable projection operator $ P \in L( \R^d ) $ where $d,m\in\N$;
Theorem~\ref{thm:intro} is applied to a finite-dimensional approximation
of the exact solution of the SEE of which convergence in probability 
is known)\label{i:problem3}, and
\item
strong convergence rates for
\emph{small noise perturbations} 
of solutions of deterministic 
differential equations
(choose $ \sigma = 0 $, $ A_t = \mu( t, Y_t ) $ 
and $ B_t = \varepsilon \, \tilde{\sigma} ( t, Y_t ) $
for all $ t \in [0,T] $
where 
$
\tilde{\sigma} \colon [0,T]\times\R^d \to \R^{d\times m} 
$
is a suitable Borel measurable function and
where $ \varepsilon > 0 $ is a sufficiently small 
parameter).
\label{i:problem4} 
\end{enumerate} 
In the literature, nearly all estimates of perturbation errors
exploit the popular \emph{global monotonicity} assumption which, in the notation of Theorem~\ref{thm:intro}, 
assumes existence of a real number $c\in\R$ such that for all $x,y\in\R^d$, $t\in[0,T]$ it holds that
\begin{equation} \begin{split} \label{eq:glob_mon} 
\langle x-y,\mu(t,x)-\mu(t,y)\rangle_{\R^d} + \tfrac{1}{2} \|\sigma(t,x)-\sigma(t,y)\|^2_{ \operatorname{HS}(\R^m,\R^d)} 
\leq c\|x-y\|_{\R^d} ^2;
\end{split} \end{equation} 
cf.\ also~\cite{HutzenthalerJentzen2020} and the references therein.
We emphasize that many SDEs from the literature do not satisfy~\eqref{eq:glob_mon} 
and that Theorem~\ref{thm:intro} does not require that the global monotonicity assumption is fulfilled.

A crucial assumption in Theorem~\ref{thm:intro} is existence of a solution of the SDE~\eqref{eq:SDE.intro}
which is twice continuously differentiable in the starting point
since in the proof of Theorem~\ref{thm:intro} we apply
It\^o's formula for independent random fields
to the random functions $\R^d\ni x\mapsto X^{x}_{t,T}\in\R^d$, $t\in[0,T]$.
This assumption is not satisfied in a number of cases. For example Li \& Scheutzow~\cite{LiScheutzow2011}
construct a two-dimensional example with smooth and globally bounded coefficient functions
which is not even strongly complete (that is, the exceptional subset of $\Omega$ where~\eqref{eq:SDE.intro} fails to hold
can not be chosen independently of the starting point); cf.\ also Hairer et al.~\cite[Theorem 1.2]{HairerHutzenthalerJentzen2015}.
Under suitable assumptions on the coefficients, however, strong completeness
and existence of a solution of~\eqref{eq:SDE.intro} which is continuous in the starting point can be ensured;
see, e.g., \cite{CoxHutzenthalerJentzen2013ArXiv,Zhang2010,Li1994}.
Existence of a solution of~\eqref{eq:SDE.intro}
which is twice continuously differentiable in the starting point
is known for strongly complete SDEs whose coefficient functions have locally Lipschitz
continuous second derivatives; see \cite[Theorem V.40]{protter05}.
In future research we
show that moments of the derivative processes up to order two are finite
if the coefficient functions grow moderately at infinity;
see \cite{HuddeHutzenthalerMazzonetto2019}.
Moreover,
in the case of non-differentiable coefficients a possible approach
is to approximate the SDE by SDEs with smooth coefficients
and to apply Theorem~\ref{thm:intro} to the sequence of smoothened SDE solutions.

We prove Theorem~\ref{thm:intro} as follows.
First, we rewrite the left-hand side of equation~\eqref{1553} as telescoping sum;
see~\eqref{eq:telescope.sum} below.
Then we apply It\^o's formula
to the random functions $\R^d\ni x\mapsto X^{x}_{t,T}\in\R^d$, $t\in[0,T]$
in order to expand the local errors. Thereby we obtain It\^o integrals
which we rewrite as Skorohod integrals
by applying Proposition~\ref{l:Skorohod.generalizes.Ito} below.
These Skorohod integrals are non-standard since the integrands are in general not
measurable with respect to a Wiener process. For this reason we introduce
an extended Skorohod integral in the appendix.
Moreover, the integrands in the It\^o integrals are adapted to different 
filtrations. We apply Proposition~\ref{l:Skorohod.on.T} below
in order to carefully rewrite the sum of these integrals as a single 
Skorohod integral.

\subsection{Notation} \label{sec:notation} 

The following notation is used throughout this article.
We denote by $ \N$ and by $ \N_0$ the sets satisfying that 
$ \N = \{ 1, 2, 3, \dots \}$
and 
$ \N_0 = \N \cup \{ 0 \}$.
For all $c \in (0, \infty)$ let $0^0$, $\tfrac{0}{0}$, $ \frac{c}{0}$, $ \frac{-c}{0}$, $0\cdot\infty$, $0\cdot(-\infty)$, $ \infty^c$
denote the extended real numbers $0^0 =1$, $\tfrac{0}{0}=0$, $ \frac{c}{0}=\infty$, $ \frac{-c}{0}=-\infty$, $0\cdot\infty=0$, $0\cdot(-\infty)=0$,
and $ \infty^c = \infty$. 
For all $T\in [0, \infty)$ let $ \Delta_T \subseteq [0, T]^2$ denote the subset with the property that 
$ \Delta_T = \{(s, t) \in[0, T]^2 \colon s \leq t \}$
and
denote by $ \nicefrac{T}{ \N}$ the set 
$ \nicefrac{T}{ \N} = \{\nicefrac{T}{n} \colon n \in \N \}$. 
For all 
$h \in (0, \infty)$, 
$r \in [0, \infty)$ 
let 
$\lceil r \rceil_h,
\lfloor r \rfloor_h, 
\lceil r \rceil_0,
\lfloor r \rfloor_0
\in [0, \infty)$ 
be the real numbers with the properties that 
$\lceil r \rceil_h = \inf \{nh \in[r, \infty) \colon n \in \N_0 \}$,
$\lfloor r \rfloor_h = \sup \{nh \in[0, r] \colon n \in \N_0 \}$,
$\lceil r \rceil_0 = r$,
and
$\lfloor r \rfloor_0 = r$.
For a real vector space $V$ and a subset $S \subseteq V$ let $ \operatorname{span}(S) \subseteq V$ denote the set with the property that 
$ \operatorname{span}(S) = \{ \sum_{i=1}^n r_i v_i \colon n \in \N, r_1, \dots, r_n \in \R, v_1, \dots, v_n \in V \}$. 
For all $(s, t) \in \Delta_T$ let $ \lambda_{[s, t]}$ be the Lebesgue-measure restricted to the Borel-sigma-algebra of $[s, t]$. 
For all $d \in \N$, $x \in \R^d$ we write $ \| x \|_{ \R^d}$ for the Euclidean norm of $x$
and for all $i \in \{ 1, \dots , d \}$ let $e^{(d)}_i$ denote the $i$-th unit vector in $ \R^d$. 
For every set $\Omega$
we denote by $\symbolsigmaalgebra(\mathcal{E})$ the smallest $\sigma$-algebra
generated by $\mathcal{E}\subseteq\mathcal{P}(\Omega)$.
For all measurable spaces $( \Omega, \mathcal F)$, $( \Omega ', \mathcal{B})$ let 
$ \mathcal M( \mathcal F, \mathcal{B})$ 
be the set 
$ \mathcal M( \mathcal F, \mathcal{B}) = \{ f \colon \Omega \to \Omega ' \colon f \textnormal{ is } \mathcal F / \mathcal{B} \textnormal{-measurable} \}$. 
For every measure space $( \Omega, \mathcal F, \mu)$, every normed vector space $(V, \| \cdot \|_V)$, and all $p \in [1, \infty)$ 
let $ \mathcal{B}(V)$ denote the Borel-sigma-algebra on $V$,
let 
$ \mathcal L^p( \mu; V)$ 
be the set with the property that
$ \mathcal L^p( \mu; V)
=
\{ f \in \mathcal M ( \mathcal F, \mathcal{B}(V)) \colon \smallint_\Omega \| f \|_V^p d \mu < \infty \}$,
let $L^p( \mu; V)$ be the set with the property that 
$
L^p( \mu, V)
=
\big \{ \{f \in \mathcal L^p( \mu, V) \colon f=g ~\mu \textnormal{-a.e.} \} \colon g \in \mathcal L^p( \mu, V) \big \},
$
and let
\begin{equation}
\| \cdot \|_{L^p( \mu; V)} 
\colon 
\big( \mathcal M( \mathcal F, \mathcal B(V)) \cup \big \{ \{f \in \mathcal M( \mathcal F, \mathcal B) \colon f=g ~\mathbb \mu\textnormal{-a.e.} \} \colon g \in \mathcal M( \mathcal F, \mathcal B) \big \} \big)
\to [0, \infty ]
\end{equation}
be the function which satisfies for all 
$f \in \big( \mathcal M( \mathcal F, \mathcal B(V)) \cup \big \{ \{h \in \mathcal M( \mathcal F, \mathcal B) \colon h=g ~\mathbb \mu\textnormal{-a.e.} \} \colon$ $g \in \mathcal M( \mathcal F, \mathcal B) \big \} \big)$
that
$ \| f \|_{ L^p( \mu; V)} = \big( \smallint_ \Omega \| f \|_V^p d \mu \big)^ \frac{1}{p}$.
For all $d, m \in \N$ and all $A \in \R^{d \times m}$ we denote by $A^* $ the transpose of $A$. 
For every measurable space $( \Omega, \mathcal F )$ and every $n\in\N$
let 
$C_b^{ \! \infty, \mathcal{F} }( \R^n \times \Omega, \R)$
be the set which satisfies that
\begin{equation}  \begin{split}
  C_b^{ \! \infty, \mathcal{F}}( \R^n \times \Omega, \R)
  =
\begin{Bmatrix}
  f\colon\R^n\times\Omega\to \R\colon & \forall \omega\in\Omega\colon f(\cdot,\omega)\in C^\infty_b(\R^n,\R),
  \\& \forall x\in\R^d\colon f(x,\cdot)\text{ is }\mathcal{F}/\mathcal{B}(\R)\text{-measurable}
\end{Bmatrix}.
\end{split}     \end{equation}
For all $d,k\in\N$ we denote by $L^{(2)}(\R^d,\R^k)$ the set of bilinear functions
from $(\R^d)^2$ to $\R^k$.
 
%
%
%
%

\section{The It\^o-Alekseev-Gr\"obner formula in the deterministic case}
The following proposition, Proposition~\ref{p:IAG}, generalizes
the Alekseev-Gr\"obner formula
(cf., e.g.,\ Hairer et al.~\cite[Theorem I.14.5]{HairerNorsettWannerI})
(which is the special case $k=d$, $f=\textup{Id}_{\R^d}$ of Proposition~\ref{p:IAG})
to general test functions.
\begin{prop}[Deterministic It\^o-Alekseev-Gr\"obner formula]\label{p:IAG}
  Let $d,k\in\N$, $T\in(0,\infty)$, let $O\subseteq \R^d$
  be a non-empty open set,
  let
  $\mu\in C^{0,1}([0,T]\times O,\R^d)$,
  $Y\in C^1([0,T],O)$,
  $X_{\mycdot,\mycdot}^{\mycdot}=(X_{s,t}^x)_{s\in[0,t],t\in[0,T],x\in O}\in C(\{(s,t)\in[0,T]^2\colon s\leq t\}\times O,O)$,
  $f\in C^1(O,\R^k)$,
  and
  assume for all $s\in[0,T]$, $t\in[s,T]$, $x\in O$ that
  $
    X_{s,t}^x=x+\int_s^t\mu(r,X_{s,r}^x)\,dr.
  $
  Then
  \begin{equation}  \begin{split}
    f(X_{0,T}^{Y_0})-f(Y_T)=\int_0^T f'(X_{s,T}^{Y_s})\tfrac{\partial}{\partial x}X_{s,T}^{Y_s}\Big(\mu(s,Y_s)-\tfrac{d}{ds}Y_s\Big)\,ds.
  \end{split}     \end{equation}
\end{prop}
\begin{proof}[Proof of Proposition~\ref{p:IAG}]
   The assumptions and the fundamental theorem of calculus
   imply for all $s\in[0,T)$, $t\in[s,T]$, $x\in O$ that $([s,T]\ni u\mapsto X_{s,u}^x\in O)\in C^{1}([s,T],O)$
   and that
   $\tfrac{\partial}{\partial t}X_{s,t}^x =\mu(t,X_{s,t}^x)$.
   This, the assumptions,
   and
   Hairer et al.~\cite[Theorem I.14.3]{HairerNorsettWannerI}) prove 
   that
   for all $s\in[0,T]$, $t\in[s,T]$
   it holds that
   $(O\ni x\mapsto X_{s,t}^x\in O)\in C^{1}(O,O)$
   and that $\tfrac{\partial}{\partial x}X_{\mycdot,\mycdot}^{\mycdot}\in C(\{(s,t)\in[0,T]^2\colon s\leq t\}\times O,L(\R^d,\R^d))$.
   Moreover, the assumptions,
   and
   Hairer et al.~\cite[Theorem I.14.4]{HairerNorsettWannerI}) show 
   that
   for all $x\in O$
   it holds that
   $([0,T]\ni s\mapsto X_{s,T}^x\in O)\in C^{1}([0,T],O)$,
   that $\tfrac{\partial}{\partial s}X_{\mycdot,T}^{\mycdot}\in C([0,T]\times O,\R^d)$,
   and that
   for all $s\in[0,T]$, $x\in O$ it holds that
   \begin{equation}  \begin{split}\label{eq:partial.initial}
     \tfrac{\partial}{\partial s}X_{s,T}^x=-\tfrac{\partial}{\partial x}X_{s,T}^x\mu(s,x).
   \end{split}     \end{equation}
   Therefore, the chain rule implies that $([0,T]\ni s\mapsto X_{s,T}^{Y_s}\in O)\in C^{1}([0,T],O)$.
   Moreover, the fundamental theorem of calculus, the chain rule, and~\eqref{eq:partial.initial} yield that
   \begin{equation}  \begin{split}
    f(X_{0,T}^{Y_0})-f(Y_T)
    &= -\int_0^T \tfrac{d}{ds}\Big(f\big(X_{s,T}^{Y_s}\big)\Big)\,ds
    \\&
    = -\int_0^T f'(X_{s,T}^{Y_s})\left(\Big(\tfrac{\partial }{\partial s}X_{s,T}^{x}\Big)\Big|_{x=Y_s}+\tfrac{\partial}{\partial x}X_{s,T}^{Y_s}\tfrac{d}{ds} Y_s\right)\,ds
    \\&
    = -\int_0^T f'(X_{s,T}^{Y_s})\Big(-\tfrac{\partial}{\partial x}X_{s,T}^{Y_s}\mu(s,Y_s)+\tfrac{\partial}{\partial x}X_{s,T}^{Y_s}\tfrac{d}{ds} Y_s\Big)\,ds
    \\&
    = \int_0^T f'(X_{s,T}^{Y_s})\tfrac{\partial}{\partial x}X_{s,T}^{Y_s}\Big(\mu(s,Y_s)-\tfrac{d}{ds}Y_s\Big)\,ds.
   \end{split}     \end{equation}
   This finishes the proof of Proposition~\ref{p:IAG}.
\end{proof}

\section{The It\^o-Alekseev-Gr\"obner formula in the general case}

The following theorem, Theorem~\ref{thm:perturbation.formula},
is the main result of this article. We note that throughout this article we use notation introduced
in Subsection~\ref{sec:notation} and in the Appendix.

\begin{theorem}[It\^o-Alekseev-Gr\"obner formula]
\label{thm:perturbation.formula}
Let 
$d, m, k \in \N$, 
$T, c \in (0, \infty)$, 
$p \in(4, \infty)$, 
$q \in[0,\tfrac{p}{2}-2)$,
let 
$( \Omega, \mathcal{F}, \mathbb P)$
be a probability space, 
let
$
W \colon [0, T] \times \Omega
\rightarrow \R^m
$
be a standard Brownian motion,
let $ \mathcal N = \{ A\in\mathcal{F}\colon \mathbb{P}(A)=0\}$,
let $\mathbb{F}=(\mathbb{F}_t)_{t\in[0,T]}$ be a filtration on $(\Omega,\mathcal{F})$ which satisfies
that $\mathbb{F}_0$ and $\symbolsigmaalgebra(W_s\colon s\in[0,T])$ are independent
and which satisfies
for all $t\in[0,T]$ that $\mathbb{F}_t=\symbolsigmaalgebra(\mathbb{F}_0\cup\symbolsigmaalgebra(W_s\colon s\in[0,t])\cup \mathcal{N})$,
let $O \subseteq \R^d$ be a non-empty open set, 
let
$ \mu \colon [0, T] \times O \to \R^d$,
$ \sigma \colon [0, T] \times O \to \R^{d\times m}$
be continuous functions,
let
$X_{ \mycdot, \mycdot}^ {\mycdot}  \colon \Delta_T \times O \times \Omega \to O$, 
$X_{ \mycdot, T}^{1, \mycdot}  \colon [0,T] \times O \times \Omega \to L( \R^d, \R^d)$,
and
$X_{ \mycdot, T}^{2, \mycdot}  \colon [0,T] \times O \times \Omega \to L^{(2)}( \R^d, \R^d)$
be continuous random fields,
assume that for all $s\in[0,T]$, $\omega\in\Omega$
it holds that
$(O \ni x \mapsto X_{s, T}^x( \omega) \in O)\in C^{2}( O, O)$, 
assume that for all 
$s\in[0,T]$, $x\in O$
the stochastic process
$[s, T] \times \Omega \ni (t, \omega) \mapsto X_{s, t}^x \in O$
is
$(\mathbb{F}_t)_{t\in[s,T]}$-adapted,
assume that for all $s\in[0,T]$,
$t \in [s, T]$, $x\in O$ it holds $ \mathbb P$-a.s.\ that 
\begin{equation} \label{1491}
\begin{split} 
X_{s, t}^x
&=
x + \smallint_s^t \mu(r, X_{s, r}^x) \,dr + \smallint_s^t \sigma(r, X_{s, r}^x)\,dW_r,
\end{split}
\end{equation}
assume that for all $(s,t) \in \Delta_T$, $x \in O$  it holds $ \mathbb P$-a.s.\ that 
$X_{t, T}^{X_{s, t}^x} = X_{s, T}^x$,
assume that for all
$(s, x, \omega) \in [0,T] \times O \times \Omega$
it holds that
$X_{s, T}^{1, x}( \omega) = \frac{ \partial}{ \partial x }\big(X_{s, T}^x( \omega)\big)$
and
$X_{s, T}^{2, x}( \omega) = \frac{ \partial^2}{ \partial x^2 }\big(X_{s, T}^x( \omega)\big)$,
let
$Y \in \mathcal L^p( \lambda_{[0, T]} \otimes \mathbb P; O)$,
$A \in \mathcal L^p( \lambda_{[0, T]} \otimes \mathbb P; \R^d)$,
$B \in \mathcal L^p( \lambda_{[0, T]} \otimes \mathbb P; \R^{d\times m})$
be stochastic processes,
assume that $Y$ has continuous sample paths, 
assume that $Y$ and $B$ are $\mathbb{F}$-predictable,
assume that
for all $t \in [0, T]$ it holds $ \mathbb P$-a.s.\ that 
\begin{equation} \label{30081}
Y_t
=
Y_0 + \smallint_0^t A_s \,ds + \smallint_0^t B_s \,dW_s,
\end{equation}
assume that
\begin{equation}\label{442b} 
\sup_{h \in \nicefrac{T}{ \N} } 
\mathbb E \bigg [
\smallint_0^T \Big \| \mu\Big(t, X_{ \lfloor t \rfloor _h, t} ^{Y_{ \lfloor t \rfloor _h } } \Big) \Big \|_{ \R^d } ^p 
+
\Big \| \sigma \Big(t, X_{ \lfloor t \rfloor _h, t} ^{Y_{ \lfloor t \rfloor _h } } \Big) \Big \|_{ \operatorname{HS}(\R^m,\R^d) } ^p \,dt
\bigg]
< \infty,
\end{equation} 
assume that
\begin{equation} \label{inthemoodforlove1}
\sup_{ \substack{r, s, t \in [0, T] \\ r \leq s \leq t}}
\mathbb E \bigg[ \Big \| X_{t, T}^{X_{r, s}^{Y_r}} \Big \|_{ \R^d}^p
+
\Big \| X_{t, T}^{1, X_{r, s}^{Y_r}} \Big \|_{L(\R^d, \R^d)}^{ \! \frac{4p}{p - 2(q + 2)}} 
+
\Big \|
X_{t, T}^{2, X_{r, s}^{Y_r}} 
\Big \|_{ L^{(2)}( \R^d, \R^d)}^{ \! \frac{2p}{p - 2(q + 2)}} \bigg]
< 
\infty , 
\end{equation}	
and let
$f \in C^2(O, \R^k)$ 
satisfy that for all $x \in O$ 
it holds that 
\begin{equation} \label{442}
\max \Big \{
\tfrac{ \| f(x) \|_{ \R^k}}{1 + \| x \|_{ \R^d}},
\| f'(x) \|_{L( \R^d, \R^k)}, 
\|f''(x) \|_{L^{(2)}( \R^d, \R^k)}
\Big \}
\leq
c (1 + \| x \|_{ \R^d }^q).
\end{equation}
Then
the stochastic process
$
\big(
f' \big(X_{r, T}^{Y_r} \big)X_{r, T}^{1, Y_r}
( \sigma(r, Y_r) - B_r)
\big)_{r \in [0, T]}
$
is Skorohod-integrable	and
it holds $ \mathbb P$-a.s.\ that 
\begin{equation}
\begin{split} \label{eq:l.perturbation.formula}
&f \big(X_{0, T}^{Y_0} \big)
- 
f (Y_T )
=
\smallint_0^T
f' \big(X_{r, T}^{Y_r} \big)
X_{r, T}^{1, Y_r}
\Big( \mu(r, Y_r) - A_r \Big) \,dr
+ 
\smallint_0^T
f' \big(X_{r, T}^{Y_r} \big)
X_{r, T}^{1, Y_r}
\Big( \sigma(r, Y_r) - B_r \Big) 
\,\delta W_r^{\mathbb{F}_0}
\\
& 
+ 
\tfrac{1}{2} 
\sum_{i, j=1}^d
\smallint_0^T
\Big( \sigma(r, Y_r) [\sigma(r, Y_r)]^*
- 
B_r [B_r]^*
\Big)_{i, j}
\Big(
f'' \big(X_{r, T}^{Y_r} \big)
\big(X_{r, T}^{1, Y_r}, X_{r, T}^{1, Y_r} \big)
+ 
f' \big(X_{r, T}^{Y_r} \big)
X_{r, T}^{2, Y_r}
\Big)
\big( e_i^{(d)}, e_j^{(d)} \big)
\,dr.
\end{split}
\end{equation} 
\end{theorem}

\begin{proof}[Proof of Theorem~\ref{thm:perturbation.formula}.]

The fact that for all $ \omega \in \Omega$ the function 
$O \ni x \mapsto X_{T, T}^x ( \omega) \in O$ is continuous and equation~\eqref{1491} imply that it holds $ \mathbb P$-a.s.\ that 
$X_{T, T}^{Y_T} = Y_T$. 
Moreover, we rewrite the left-hand side of equation~\eqref{eq:l.perturbation.formula} as telescoping sum and obtain that for all $n \in \N$, $h \in \{ \tfrac{T}{n} \}$ it holds $ \mathbb P$-a.s.\ that
\begin{equation} \label{eq:telescope.sum}
\begin{split} 
&f \big(X_{0, T}^{Y_{0}} \big)
- 
f \big(Y_T \big)
=
f \big(X_{0h, T}^{Y_{0h}} \big)
- 
f \big(X_{nh, T}^{Y_{nh}} \big)
=
\sum_{i=0}^{n - 1}
\Big(
f \big(
X_{ih, T}^{Y_{ih}}
\big)
- 
f \big(
X_{(i + 1)h, T}^{Y_{(i + 1)h}}
\big)
\Big)
\\&
=
\sum_{i=0}^{n - 1}
\Big(
f \big(
X_{ih, T}^{Y_{ih}}
\big)
- 
f \big(
X_{(i + 1)h, T}^{Y_{ih}}
\big)
\Big)
- 
\sum_{i=0}^{n - 1}
\Big(
f \big( X_{(i + 1)h, T}^{Y_{(i + 1)h}} \big)
- 
f \big( X_{(i + 1)h, T}^{Y_{ih}} \big)
\Big).
\end{split} 
\end{equation}

\textbf{First, we analyze the second sum on the right-hand side of equation~\eqref{eq:telescope.sum}.}
For all $t \in [0, T]$, $x \in O$, $i \in \{1, 2\}$ 
the functions 
$ \Omega \ni \omega \mapsto X_{t, T}^x ( \omega) \in O$,
$ \Omega \ni \omega \mapsto X_{t, T}^{i, x} ( \omega) \in L^{(i)}( \R^d, \R^d)$
are
$ \mathfrak{S}(\mathcal{N}\cup\mathfrak{S} (W_s-W_t \colon s \in [t, T]))$-measurable.
This together with the fact that for all 
$ \omega \in \Omega$, 
$t \in [0, T]$ 
it holds that 
$ \Big( O \ni x \mapsto f(X_{t, T}^x ( \omega)) \in \R^k \Big) \in C^2 (O, \R^k)$
implies that for all $t \in [0, T]$ the function
$ \Omega \ni \omega \mapsto \Big(O \ni x \mapsto f(X_{t, T}^x ( \omega)) \in \R^k \Big) \in C^2 (O, \R^k)$
is independent of the sigma-algebra $ \mathbb{F}_t $. 
It\^o's formula for independent random fields
(e.g., Klenke~\cite[Theorem 25.30 and Remark 25.26]{Klenke2008}) 
(applied with the functions
$ \Omega \ni \omega \mapsto \Big(O \ni x \mapsto f(X_{(i + 1)h, T}^x ( \omega)) \in \R^k \Big)\in C^2(O,\R^k)$ 
for $n\in\N$, $i \in \{0, 1, \dots, n - 1 \}$, $h \in \{ \tfrac{T}{n} \}$)
yields that
for all $n \in \N$, $i \in \{0, 1, \dots, n - 1 \}$, $h \in \{ \tfrac{T}{n} \}$
it holds $ \mathbb P$-a.s.\ that
\begin{equation} \label{eq:equality.for.last.summand} 
\begin{split} 
&
f \big( X_{(i + 1)h, T}^{Y_{(i + 1)h}} \big) 
- 
f \big( X_{(i + 1)h, T}^{Y_{ih}} \big) 
\\&=
\smallint_{ih}^{(i + 1)h}
\tfrac{ \partial}{ \partial x} \big( f(X_{(i + 1)h, T}^{x}) \big) \big|_{x=Y_r}
\,dY_r
+ \tfrac{1}{2}
\sum_{l, j=1}^d
\smallint_{ih}^{(i + 1)h}
\tfrac{ \partial^2}{ \partial x^2} \big( f(X_{(i + 1)h, T}^{x}) \big) \big|_{x=Y_r}
(e_l^{(d)}, e_j^{(d)})
\,d \left( \langle Y \rangle_r \right)_{l, j}
\\ &
=
\smallint_{ih}^{(i + 1)h}
f' \big(X_{ (i + 1)h , T}^{Y_r} \big)
X_{ (i + 1)h , T}^{1, Y_r} A_r \,dr 
+ 
\smallint_{ih}^{(i + 1)h}
f' \big(X_{ (i + 1)h , T}^{Y_r} \big)
X_{ (i + 1)h , T}^{1, Y_r} B_r \,dW_r 
\\ &
\quad + \tfrac{1}{2}
\sum_{l, j = 1}^d
\smallint_{ih}^{(i + 1)h}
\left(B_r [B_r]^* \right)_{l, j}
\Big( f''\big( X_{ (i + 1)h , T}^{Y_r} \big) 
\big( X_{ (i + 1)h , T}^{1, Y_r}, X_{ (i + 1)h , T}^{1, Y_r} \big)
+
f'\big( X_{ (i + 1)h , T}^{Y_r} \big) X_{ (i + 1)h , T}^{2, Y_r} \Big)
(e_l^{(d)}, e_j^{(d)})
\,dr.
\end{split}
\end{equation}
Inequalities~\eqref{442} and~\eqref{inthemoodforlove1} imply 
for all $i \in \{1, 2 \}$  that
\begin{equation} \label{443}
\begin{split}
\sup_{ \substack{r, s, t \in [0, T] \\ r \leq s \leq t}}
\Big \|
f^{(i)} 
\big(X_{ t, T}^{X_{r, s}^{Y_r}} \big) 
\Big \|_{L^ \frac{p}{q} ( \mathbb P; L^{(i)}( \R^d, \R^k))} 
&
\leq
c
\sup_{ \substack{r, s, t \in [0, T] \\ r \leq s \leq t}}
\Big \|
1
+ 
\Big \| 
X_{t, T}^{X_{r, s}^{Y_r}} 
\Big \|_{ \R^d}^q
\Big \|_{ L^ \frac{p}{q} ( \mathbb P; \R)}
\\&
\leq
c
\Bigg( 
1
+ 
\sup_{ \substack{r, s, t \in [0, T] \\ r \leq s \leq t}}
\Big \|
X_{t, T}^{X_{r, s}^{Y_r}} 
\Big \|_{ L^p ( \mathbb P; \R^d)}^q
\Bigg)
< \infty. 
\end{split}
\end{equation}
H\"older's inequality, inequalities~\eqref{inthemoodforlove1},~\eqref{443}, and the assumption 
$B \in \mathcal L^p( \lambda_{[0, T]} \otimes \mathbb P; \R^{d\times m})$ 
imply that for all $n \in \N$, $i \in \{0, 1, \dots, n - 1 \}$, $h \in \{ \tfrac{T}{n} \}$ it holds that 
\begin{align} \label{238}
&
\Big \|
f' \big(X_{(i + 1)h , T}^{Y_{ \mycdot}} \big) X_{(i + 1)h , T}^{1, Y_{ \mycdot}} B_ \cdot
\Big \|_{L^2( \mathbb{P}; L^2(\lambda_{[ih, (i + 1)h]} ; \R^{k\times m}))}
\leq
\Big \|
f' \big(X_{\lceil \mycdot\rceil_h , T}^{Y_{ \mycdot}} \big) X_{\lceil \mycdot\rceil_h , T}^{1, Y_{ \mycdot}} B_ \cdot
\Big \|_{L^2( \lambda_{[0,T]} \otimes \mathbb P; \R^{k\times m})}
\nonumber
\\&
\leq
\Big \|
\big \| 
f' \big(X_{\lceil \mycdot\rceil_h , T}^{Y_{ \mycdot}} \big) 
\big \|_{L( \R^d, \R^k)} 
\big \|
X_{\lceil \mycdot\rceil_h , T}^{1, Y_{ \mycdot}} 
\big \|_{L( \R^d, \R^d)} 
\|
B_ \cdot
\|_{ \operatorname{HS}(\R^m,\R^d)} 
\Big \|_{L^2( \lambda_{[0,T]} \otimes \mathbb P; \R)}
\\&
\leq
\big \|
f' \big(X_{\lceil \mycdot\rceil_h , T}^{Y_{ \mycdot}} \big) 
\big \|_{L^{ \! \frac{p}{q}}( \lambda_{[0,T]} \otimes \mathbb P; L( \R^d, \R^k))} 
\big \|
X_{\lceil \mycdot\rceil_h , T}^{1, Y_{ \mycdot}} 
\big \|_{L^{ \! \frac{2p}{p - 2(q + 1)}}( \lambda_{[0,T]} \otimes \mathbb P; L( \R^d, \R^d))} 
\|
B
\|_{L^p( \lambda_{[0,T]} \otimes \mathbb P; \R^{d\times m})}
\nonumber
\\&
\leq
T^{  \frac{p - 2}{2p}}
\Big(
\sup_{(r, s) \in \Delta_T} 
\big \|
f' \big(X_{s, T}^{Y_r} \big) 
\big \|_{L^{ \! \frac{p}{q}}( \mathbb P; L( \R^d, \R^k))} 
\Big)
\Big(
\sup_{(r, s) \in \Delta_T} 
\big \|
X_{s, T}^{1, Y_r} 
\big \|_{L^{ \! \frac{2p}{p - 2(q + 1)}}( \mathbb P; L( \R^d, \R^d))} 
\Big)
\|
B
\|_{L^p( \lambda_{[0, T]} \otimes \mathbb P; \R^{d\times m})}
< \infty.
\nonumber
\end{align}
For all $n \in \N$, $i \in \{0, 1, \dots, n - 1 \}$, $h \in \{ \tfrac{T}{n} \}$ the stochastic process
$
\big( f' \big(X_{(i + 1)h, T}^{Y_r} \big)
X_{(i + 1)h, T}^{1, Y_r} B_r
\big)_{r \in [ih, (i + 1)h]}
$
is predictable with respect to the filtration
\begin{equation}  \begin{split}\label{eq:filtration}
 \big( \mathfrak{S} \big(\mathbb{F}_r\cup\mathfrak{S} \big( \{ W_s - W_{(i + 1)h} \colon s \in [(i + 1)h, T] \} \big)\big) \big)_{r \in[ih, (i + 1)h]}. 
\end{split}     \end{equation}
Proposition~\ref{l:Skorohod.generalizes.Ito} together with
inequality~\eqref{238}, Proposition~\ref{l:Skorohod.on.T},
and linearity of the Skorohod integral
yield that for all $h\in T/\N$ it holds that
$ (f' \big(X_{\lceil r\rceil_h, T}^{Y_r} \big) X_{\lceil r\rceil_h, T}^{1, Y_r} B_r)_{r\in[0,T]} $
is Skorohod-integrable
and that
for all $n \in \N$, $h \in \{ \tfrac{T}{n} \}$ it holds $ \mathbb P$-a.s.\ that
\begin{equation} \label{hier}
\begin{split}
&
\sum_{i=0}^{n - 1}
\smallint_{ih}^{(i + 1)h}
f' \big(X_{(i + 1)h, T}^{Y_r} \big)
X_{(i + 1)h, T}^{1, Y_r} B_r \,dW_r
\\&
=
\sum_{i=0}^{n - 1}
\smallint_{ih}^{(i + 1)h}
f' \big(X_{(i + 1)h, T}^{Y_r} \big)
X_{(i + 1)h, T}^{1, Y_r} B_r
\,\delta W_r^{\mathfrak{S}( \mathbb{F}_{ih} \cup\mathfrak{S}( \{ W_s - W_{(i + 1)h} \colon s \in [(i + 1)h, T] \}) ) }
\\&
=
\sum_{i=0}^{n - 1}
\smallint_0^T
\1_{[ih, (i + 1)h]}(r)
f' \big(X_{ \lceil r \rceil_h , T}^{Y_r} \big)
X_{ \lceil r \rceil_h , T}^{1, Y_r} B_r 
\, \delta W_r^{ \mathbb{F}_0}
\\&
=
\smallint_0^T
f' \big(X_{ \lceil r \rceil_h , T}^{Y_r} \big)
X_{ \lceil r \rceil_h , T}^{1, Y_r} B_r 
\, \delta  W_r^{\mathbb{F}_0}.
\end{split}
\end{equation}
Equations~\eqref{eq:equality.for.last.summand} and~\eqref{hier} imply that for all $n \in \N$, $h \in \{ \tfrac{T}{n} \}$ it holds $ \mathbb P$-a.s.\ that
\begin{equation} \label{grts}
\begin{split} 
&\sum_{i=0}^{n - 1}
\Big(
f \big(
X_{(i + 1)h, T}^{Y_{(i + 1)h}}
\big)
- 
f \big(
X_{(i + 1)h, T}^{Y_{ih}}
\big)
\Big)
\\&
=
\smallint_0^T
f' \big(X_{ \lceil r \rceil_h , T}^{Y_r} \big)
X_{ \lceil r \rceil_h , T}^{1, Y_r} A_r \,dr
+ 
\smallint_0^T
f' \big(X_{ \lceil r \rceil_h , T}^{Y_r} \big)
X_{ \lceil r \rceil_h , T}^{1, Y_r} B_r
\,\delta
W_r^{\mathbb{F}_0}
\\&
\quad + \tfrac{1}{2}
\sum_{l, j = 1}^d
\smallint_0^T
\left(B_r [B_r]^* \right)_{l, j}
\Big( f''\big( X_{ \lceil r \rceil_h , T}^{Y_r} \big) \big( X_{ \lceil r \rceil_h , T}^{1, Y_r}, X_{ \lceil r \rceil_h , T}^{1, Y_r} \big)
+
f'\big( X_{ \lceil r \rceil_h , T}^{Y_r} \big) X_{ \lceil r \rceil_h , T}^{2, Y_r} \Big)
(e_l^{(d)}, e_j^{(d)})
\,dr.
\end{split}
\end{equation}
\textbf{Next we analyze the first sum on the right-hand side of equation~\eqref{eq:telescope.sum}.}
For all $(s, t) \in \Delta_T$, $x \in O$ it holds that 
$
\mathbb P
\big(X_{s, T}^x = X_{t, T}^{X_{s, t}^x} \big) =1
$. 
This and the fact that $X$ is a continuous random field
imply for all 
$(s, t) \in \Delta_T$ that 
$ \mathbb P \Big( X_{s, T}^{Y_s} = X_{t, T}^{ \! X_{s, t}^{Y_s}} \Big)=1$.
For all $t \in [0, T]$, $x \in O$, $i \in \{1, 2\}$ 
the functions 
$ \Omega \ni \omega \mapsto X_{t, T}^x ( \omega) \in O$,
$ \Omega \ni \omega \mapsto X_{t, T}^{i, x} ( \omega) \in L^{(i)}( \R^d, \R^d)$
are
$\mathfrak{S}(\mathcal{N}\cup \mathfrak{S} (W_s -W_t\colon s \in [t, T]))$-measurable.
This together with the fact that for all 
$ \omega \in \Omega$, 
$t \in [0, T]$ 
it holds that 
$ \Big(O \ni x \mapsto f(X_{t, T}^x ( \omega)) \in \R^k) \Big) \in C^2 (O, \R^k)$
implies that for all $t \in [0, T]$ the function
$ \Omega \ni \omega \mapsto \Big(O \ni x \mapsto f(X_{t, T}^x ( \omega)) \in \R^k) \Big)\in C^2(O,\R^k)$
is independent of the sigma-algebra
$\mathbb{F}_t$. 
It\^o's formula for independent random fields
(e.g., Klenke~\cite[Theorem 25.30 and Remark 25.26]{Klenke2008}) 
(applied with the functions
$ \Omega \ni \omega \mapsto \Big( O \ni x \mapsto f(X_{(i + 1)h, T}^x ( \omega)) \in \R^k \Big)\in C^2(O,\R^k)$
for 
$n \in \N$, $i \in \{0, 1, \dots, n - 1 \}$, $h \in \{ \tfrac{T}{n} \}$)
yields that
for all $n \in \N$, $i \in \{0, 1, \dots, n - 1 \}$, $h \in \{ \tfrac{T}{n} \}$
it holds $ \mathbb P$-a.s.\ that
\begin{equation} \label{242}
\begin{split}
&f\big( X_{ih, T}^{Y_{ih}} \big)
- 
f \big( X_{(i + 1)h, T}^{Y_{ih}} \big)
=
f \Big(X_{(i + 1)h, T}^{X_{ih, (i + 1)h}^{Y_{ih}}} \Big) 	
- 
f \big( X_{(i + 1)h, T}^{Y_{ih}} \big) 
\\&
= 
\smallint_{ih}^{(i + 1)h} f' \Big(X_{(i + 1)h, T}^{X_{ih, r}^{Y_{ih}}} \Big) X_{(i + 1)h, T}^{1, X_{ih, r}^{Y_{ih}}} \,d X_{ih, r}^{Y_{ih}} 
\\& 
\quad 
+ 
\tfrac{1}{2} \sum_{l, j = 1}^d \smallint_{ih}^{(i + 1)h} 
\bigg( f'' \bigg( X_{(i + 1)h, T}^{X_{ih, r}^{Y_{ih}}} \bigg) 
\bigg( X_{(i + 1)h, T}^{1, X_{ih, r}^{Y_{ih}}}, X_{(i + 1)h, T}^{1, X_{ih, r}^{Y_{ih}}} \bigg)
+ f' \bigg( X_{(i + 1)h, T}^{X_{ih, r}^{Y_{ih}}} \bigg) X_{(i + 1)h, T}^{2, X_{ih, r}^{Y_{ih}}} \bigg)
\big( e^{(d)}_l, e^{(d)}_j \big)
\,d \big( \langle X_{ih, \mycdot}^{Y_{ih}} \rangle_r \big)_{l, j} 
\\&
= 
\smallint_{ih}^{(i + 1)h} 
f' \bigg(X_{ \lceil r \rceil_h, T}^{X_{ \lfloor r \rfloor_h, r}^{Y_{ \lfloor r \rfloor_h}}} \bigg) 
X_{ \lceil r \rceil_h, T}^{1, X_{ \lfloor r \rfloor_h, r}^{Y_{ \lfloor r \rfloor_h}}} 
\mu \big( r, X_{\lfloor r\rfloor_h, r}^{Y_{\lfloor r\rfloor_h}} \big) \,dr
+ 
\smallint_{ih}^{(i + 1)h} f' \bigg(X_{ \lceil r \rceil_h, T}^{X_{ \lfloor r \rfloor_h, r}^{Y_{ \lfloor r \rfloor_h}}} \bigg) 
X_{ \lceil r \rceil_h, T}^{1, X_{ \lfloor r \rfloor_h, r}^{Y_{ \lfloor r \rfloor_h}}} 
\sigma\big( r, X_{\lfloor r\rfloor_h, r}^{Y_{\lfloor r\rfloor_h}} \big) \,dW_r 
\\&
\quad 
+ \tfrac{1}{2} 
\sum_{l, j = 1}^d \smallint_{ih}^{(i + 1)h} 
\Big( \sigma \Big(r, X_{ \lfloor r \rfloor_h, r}^{Y_{ \lfloor r \rfloor_h}} \Big) \Big[\sigma \Big(r, X_{ \lfloor r \rfloor_h, r}^{Y_{ \lfloor r \rfloor_h}} \Big)\Big]^* \Big)_{l, j}
\\&
\quad \quad \quad \quad
\cdot
\bigg( f'' \bigg( X_{ \lceil r \rceil_h, T}^{X_{ \lfloor r \rfloor_h, r}^{Y_{ \lfloor r \rfloor_h}}} \bigg) 
\bigg(X_{ \lceil r \rceil_h, T}^{1, X_{ \lfloor r \rfloor_h, r}^{Y_{ \lfloor r \rfloor_h}}}, X_{ \lceil r \rceil_h, T}^{1, X_{ \lfloor r \rfloor_h, r}^{Y_{ \lfloor r \rfloor_h}}} \bigg)
+ f' \bigg( X_{ \lceil r \rceil_h, T}^{X_{ \lfloor r \rfloor_h, r}^{Y_{ \lfloor r \rfloor_h}}} \bigg) 
X_{ \lceil r \rceil_h, T}^{2, X_{ \lfloor r \rfloor_h, r}^{Y_{ \lfloor r \rfloor_h}}}
\bigg)
\big( e^{(d)}_l, e^{(d)}_j \big)
\,dr.
\end{split} 
\end{equation}
H\"{o}lder's inequality and inequalities~\eqref{443},~\eqref{inthemoodforlove1},~\eqref{442b} imply that for all $n \in \N$, $i \in \{0, 1, \dots, n - 1 \}$, $h \in \{ \tfrac{T}{n} \}$ it holds that
\begin{align} \label{245}
&
\nonumber
\Bigg \|
f' \Bigg(X_{\lceil \mycdot\rceil_h , T}^{X_{ \lfloor \mycdot \rfloor_{h}, \mycdot}^{Y_{ \lfloor \mycdot \rfloor_{h}}}} \Bigg) 
X_{\lceil \mycdot\rceil_h , T}^{1, X_{ \lfloor \mycdot \rfloor_{h}, \mycdot}^{Y_{ \lfloor \mycdot \rfloor_{h}}}} 
\sigma \Big( \cdot, X_{ \lfloor \mycdot \rfloor _h, \mycdot }^{Y_{ \lfloor \mycdot \rfloor _h }} \Big)
\Bigg \|_{L^2(\mathbb{P}; L^2( \lambda_{[ih, (i + 1)h]} ; \R^{k\times m}))} 
\\&\nonumber
\leq
\Bigg \|
f' \Bigg(X_{\lceil \mycdot \rceil_h , T}^{X_{ \lfloor \mycdot \rfloor_{h}, \mycdot}^{Y_{ \lfloor \mycdot \rfloor_{h}}}} \Bigg) 
X_{\lceil \mycdot \rceil_h , T}^{1, X_{ \lfloor \mycdot \rfloor_{h}, \mycdot}^{Y_{ \lfloor \mycdot \rfloor_{h}}}} 
\sigma \Big( \cdot, X_{ \lfloor \mycdot \rfloor _h, \mycdot }^{Y_{ \lfloor \mycdot \rfloor _h }} \Big)
\Bigg \|_{L^2( \lambda_{[0,T]} \otimes \mathbb P; \R^{k\times m})} 
\\&
\leq
\Bigg \|
\Bigg \| 
f' \Bigg(X_{\lceil \mycdot \rceil_h , T}^{X_{ \lfloor \mycdot \rfloor_{h}, \mycdot}^{Y_{ \lfloor \mycdot \rfloor_{h}}}} \Bigg) 
\Bigg \|_{L( \R^d, \R^k)} 
\Bigg \|
X_{\lceil \mycdot \rceil_h , T}^{1, X_{ \lfloor \mycdot \rfloor_{h}, \mycdot}^{Y_{ \lfloor \mycdot \rfloor_{h}}}} 
\Bigg \|_{L( \R^d, \R^d)} 
\Big \|
\sigma \Big( \cdot, X_{ \lfloor \mycdot \rfloor _h, \mycdot }^{Y_{ \lfloor \mycdot \rfloor _h }} \Big)
\Big \|_{ \operatorname{HS}(\R^m,\R^d)} 
\Bigg \|_{L^2( \lambda_{[0,T]} \otimes \mathbb P; \R)}
\\&
\nonumber
\leq
\Bigg \|
f' \Bigg(X_{\lceil \mycdot \rceil_h , T}^{X_{ \lfloor \mycdot \rfloor_{h}, \mycdot}^{Y_{ \lfloor \mycdot \rfloor_{h}}}} \Bigg) 
\Bigg \|_{L^{ \! \frac{p}{q}}( \lambda_{[0,T]} \otimes \mathbb P; L( \R^d, \R^k))} 
\Bigg \|
X_{\lceil \mycdot \rceil_h , T}^{1, X_{ \lfloor \mycdot \rfloor_{h}, \mycdot}^{Y_{ \lfloor \mycdot \rfloor_{h}}}} 
\Bigg \|_{L^{ \! \frac{2p}{p - 2(q + 1)}}( \lambda_{[0,T]} \otimes \mathbb P; L( \R^d, \R^d))} 
\Big \|
\sigma \Big( \cdot, X_{ \lfloor \mycdot \rfloor _h, \mycdot }^{Y_{ \lfloor \mycdot \rfloor _h }} \Big)
\Big \|_{L^p( \lambda_{[0,T]} \otimes \mathbb P; \R^{d\times m})}
\\&
\leq
T^{ \! \frac{p - 2}{2p}}
\Bigg(
\sup_{ \substack{r, s, t \in [0, T] \\ r \leq s \leq t}} 
\bigg \| 
f' \bigg(X_{t, T}^{X_{r, s}^{Y_r}} \bigg) 
\bigg \|_{L^{ \! \frac{p}{q}}( \mathbb P; L( \R^d, \R^k))} 
\Bigg)
\Bigg( 
\sup_{ \substack{r, s, t \in [0, T] \\ r \leq s \leq t}} 
\bigg \|
X_{t, T}^{1, X_{r, s}^{Y_r}} 
\bigg \|_{L^{ \! \frac{2p}{p - 2(q + 1)}}( \mathbb P; L( \R^d, \R^d))} 
\Bigg)
\nonumber
\\&
\quad \quad \quad \quad \quad \quad \quad \quad \quad \quad \quad \quad \quad \quad \quad \quad \quad \quad \quad \quad
\cdot
\Big(
\sup_{ \kappa \in \nicefrac{T}{N}}
\Big \|
\sigma \Big( \cdot, X_{ \lfloor \mycdot \rfloor _ \kappa, \mycdot }^{Y_{ \lfloor \mycdot \rfloor _ \kappa }} \Big)
\Big \|_{L^p( \lambda_{[0, T]} \otimes \mathbb P; \R^{d\times m})}
\Big)
< \infty.
\nonumber
\end{align}
For all $n \in \N$, $i \in \{0, 1, \dots, n - 1 \}$, $h \in \{ \tfrac{T}{n} \}$ the process
$
\Big(
f' \Big(X_{ (i+1)h, T}^{X_{ ih, r}^{Y_{ ih}}} \Big) 
X_{ (i+1)h, T}^{1, X_{ ih, r}^{Y_{ ih }}} 
\sigma \Big( r, X_{ ih, r}^{Y_{ ih }} \Big)
\Big)_{r \in [ih, (i + 1)h]}
$
is predictable with respect to the filtration~\eqref{eq:filtration}.
Proposition~\ref{l:Skorohod.generalizes.Ito} together with inequality~\eqref{245},
Proposition~\ref{l:Skorohod.on.T},
and linearity of the Skorohod integral
assert that the process
$f' \Big(X_{ \lceil \mycdot \rceil_h, T}^{X_{ \lfloor \mycdot \rfloor_h, \mycdot}^{Y_{ \lfloor \mycdot \rfloor_h}}} \Big) 
X_{ \lceil \mycdot \rceil_h, T}^{1, X_{ \lfloor \mycdot \rfloor_h, \mycdot}^{Y_{ \lfloor \mycdot \rfloor_h }}} 
\sigma \Big( \cdot, X_{ \lfloor \mycdot \rfloor _h, \mycdot}^{Y_{ \lfloor \mycdot \rfloor _h }} \Big) 
$
is Skorohod-integrable
and that
for all $n \in \N$, $h \in \{ \tfrac{T}{n} \}$ it holds $ \mathbb P$-a.s.\ that
\begin{equation} \label{hitthelights}
\begin{split}
& 
\sum_{i=0}^{n - 1} 
\smallint_{ih}^{(i + 1)h} 
f' \Bigg( X_{ \lceil r \rceil_h, T}^{X_{ \lfloor r \rfloor_h, r}^{Y_{ \lfloor r \rfloor_h}}} \Bigg) 
X_{ \lceil r \rceil_h, T}^{1, X_{ \lfloor r \rfloor_h, r}^{Y_{ \lfloor r \rfloor_h}}} 
\sigma \Big( r, X_{ \lfloor r \rfloor _h, r}^{Y_{ \lfloor r \rfloor _h }} \Big) 
\,dW_r
\\&
= 
\sum_{i=0}^{n - 1} 
\smallint_{ih}^{(i + 1)h} 
f' \Bigg( X_{ \lceil r \rceil_h, T}^{X_{ \lfloor r \rfloor_h, r}^{Y_{ \lfloor r \rfloor_h}}} \Bigg) 
X_{ \lceil r \rceil_h, T}^{1, X_{ \lfloor r \rfloor_h, r}^{Y_{ \lfloor r \rfloor_h}}} 
\sigma \Big( r, X_{ \lfloor r \rfloor _h,r }^{Y_{ \lfloor r \rfloor _h }} \Big) 
\, \delta W_r^{ \mathfrak{S}(\mathbb{F}_{ih}\cup\mathfrak{S}(W_s-W_{(i+1)h}\colon s\in[(i+1)h,T]))}
\\&
= \sum_{i=0}^{n - 1}
\smallint_0^T 
\mathbbm 1_{[ih, (i + 1)h]}(r) 
f' \bigg(X_{ \lceil r \rceil_h, T}^{X_{ \lfloor r \rfloor_h, r}^{Y_{ \lfloor r \rfloor_h}}} \bigg) 
X_{ \lceil r \rceil_h, T}^{1, X_{ \lfloor r \rfloor_h, r}^{Y_{ \lfloor r \rfloor_h}}} 
\sigma \Big(r, X_{ \lfloor r \rfloor _h, r}^{Y_{ \lfloor r \rfloor _h }} \Big) 
\, \delta W_r^{\mathbb{F}_0}
\\&
= 
\smallint_0^T
f' \bigg(X_{ \lceil r \rceil_h, T}^{X_{ \lfloor r \rfloor_h, r}^{Y_{ \lfloor r \rfloor_h}}} \bigg) 
X_{ \lceil r \rceil_h, T}^{1, X_{ \lfloor r \rfloor_h, r}^{Y_{ \lfloor r \rfloor_h }}} 
\sigma \Big( r, X_{ \lfloor r \rfloor _h, r}^{Y_{ \lfloor r \rfloor _h }} \Big) 
\,\delta
W_r^{\mathbb{F}_0}.
\end{split} 
\end{equation}
Equations~\eqref{242} and~\eqref{hitthelights} imply that for all $n \in \N$, $h \in \{ \tfrac{T}{n} \}$ it holds $ \mathbb P$-a.s.\ that
\begin{equation}  \begin{split}
\label{erstersummand}
&
\sum_{i=0}^{n - 1} 
\Big(
f \big( X_{ih, T}^{Y_{ih}} \big) - f \big(X_{(i + 1)h, T}^{Y_{ih}} \big) 
\Big)
\\&
=
\smallint_0^T 
f' \bigg( X_{ \lceil r \rceil_h, T}^{X_{ \lfloor r \rfloor_h, r}^{Y_{ \lfloor r \rfloor_h}}} \bigg) 
X_{ \lceil r \rceil_h, T}^{1, X_{ \lfloor r \rfloor_h, r}^{Y_{ \lfloor r \rfloor_h }}} 
\mu \Big( r, X_{ \lfloor r \rfloor _h, r}^{Y_{ \lfloor r \rfloor _h }} \Big) 
\,dr 
+ 
\smallint_0^T 
f' \bigg(X_{ \lceil r \rceil_h, T}^{X_{ \lfloor r \rfloor_h, r}^{Y_{ \lfloor r \rfloor_h}}} \bigg) 
X_{ \lceil r \rceil_h, T}^{1, X_{ \lfloor r \rfloor_h, r}^{Y_{ \lfloor r \rfloor_h }}} 
\sigma \Big( r, X_{ \lfloor r \rfloor _h, r}^{Y_{ \lfloor r \rfloor _h }} \Big) 
\,\delta
W_r^{\mathbb{F}_0} 
\\&
+ \tfrac{1}{2} 
\sum_{l, j = 1}^d \smallint_0^T
\Big( \sigma \big(r, X_{ \lfloor r \rfloor_h, r}^{Y_{ \lfloor r \rfloor_h}} \big) \big[\sigma \big(r, X_{ \lfloor r \rfloor_h, r}^{Y_{ \lfloor r \rfloor_h}} \big)\big]^* \Big)_{l, j}
\\&
\quad \quad \quad \quad
\cdot
\Bigg( 
f'' \bigg( X_{ \lceil r \rceil_h, T}^{X_{ \lfloor r \rfloor_h, r}^{Y_{ \lfloor r \rfloor_h}}} \bigg) 
\bigg(X_{ \lceil r \rceil_h, T}^{1, X_{ \lfloor r \rfloor_h, r}^{Y_{ \lfloor r \rfloor_h}}}, X_{ \lceil r \rceil_h, T}^{1, X_{ \lfloor r \rfloor_h, r}^{Y_{ \lfloor r \rfloor_h}}} \bigg)
+ f' \bigg( X_{ \lceil r \rceil_h, T}^{X_{ \lfloor r \rfloor_h, r}^{Y_{ \lfloor r \rfloor_h}}} \bigg) 
X_{ \lceil r \rceil_h, T}^{2, X_{ \lfloor r \rfloor_h, r}^{Y_{ \lfloor r \rfloor_h}}}
\Bigg)
\big( e^{(d)}_l, e^{(d)}_j \big)
\,dr.
\end{split}     \end{equation}
Equations~\eqref{eq:telescope.sum},~\eqref{erstersummand}, and~\eqref{grts} imply that for all
$h \in \nicefrac{T}{ \N }$ it holds $ \mathbb P$-a.s.\ that 
\begin{align} \label{fastfertig}
&f \big(X_{0, T}^{Y_0} \big) - f \big(Y_T \big) 
\nonumber
\\&
=
\smallint_0^T f' \bigg(X_{ \lceil r \rceil_h, T}^{X_{ \lfloor r \rfloor_h, r}^{Y_{ \lfloor r \rfloor_h}}} \bigg) X_{ \lceil r \rceil_h, T}^{1, X_{ \lfloor r \rfloor_h, r}^{Y_{ \lfloor r \rfloor_h }}} \mu \big(r, X_{ \lfloor r \rfloor _h, r}^{Y_{ \lfloor r \rfloor _h }} \big) 
- 
f' \big(X_{ \lceil r \rceil_h , T}^{Y_r} \big)
X_{ \lceil r \rceil_h , T}^{1, Y_r} A_r \,dr
\nonumber
\\&
+ 
\smallint_0^T 
f' \bigg(X_{ \lceil r \rceil_h, T}^{X_{ \lfloor r \rfloor_h, r}^{Y_{ \lfloor r \rfloor_h}}} \bigg) X_{ \lceil r \rceil_h, T}^{1, X_{ \lfloor r \rfloor_h, r}^{Y_{ \lfloor r \rfloor_h }}} \sigma \Big(r, X_{ \lfloor r \rfloor _h, r}^{Y_{ \lfloor r \rfloor _h }} \Big) 
- 
f' \big(X_{ \lceil r \rceil_h , T}^{Y_r} \big)
X_{ \lceil r \rceil_h , T}^{1, Y_r} B_r 
\,\delta
W_r^{\mathbb{F}_0} 
\nonumber
\\&
+ 
\tfrac{1}{2} 
\sum_{l, j = 1}^d 
\smallint_0^T
\Big( \sigma \big(r, X_{ \lfloor r \rfloor_h, r}^{Y_{ \lfloor r \rfloor_h }} \big) \big[\sigma \big(r, X_{ \lfloor r \rfloor_h , r}^{Y_{ \lfloor r \rfloor_h}} \big)\big]^* \Big)_{l, j}
\\&
\quad \quad \quad \quad
\cdot
\Bigg( f'' \bigg( X_{ \lceil r \rceil_h, T}^{X_{ \lfloor r \rfloor_h, r}^{Y_{ \lfloor r \rfloor_h}}} \bigg) \bigg(X_{ \lceil r \rceil_h, T}^{1, X_{ \lfloor r \rfloor_h, r}^{Y_{ \lfloor r \rfloor_h}}}, X_{ \lceil r \rceil_h, T}^{1, X_{ \lfloor r \rfloor_h, r}^{Y_{ \lfloor r \rfloor_h}}} \bigg)
+ f' \bigg( X_{ \lceil r \rceil_h, T}^{X_{ \lfloor r \rfloor_h, r}^{Y_{ \lfloor r \rfloor_h}}} \bigg) X_{ \lceil r \rceil_h, T}^{2, X_{ \lfloor r \rfloor_h, r}^{Y_{ \lfloor r \rfloor_h}}}
\Bigg)
\big( e^{(d)}_l, e^{(d)}_j \big)
\,dr
\nonumber
\\&
- 
\tfrac{1}{2}
\sum_{l, j = 1}^d
\smallint_0^T
\left(B_r [B_r]^* \right)_{l, j}
\Big( f''\big( X_{ \lceil r \rceil_h , T}^{Y_r} \big) \big( X_{ \lceil r \rceil_h , T}^{1, Y_r}, X_{ \lceil r \rceil_h , T}^{1, Y_r} \big)
+
f'\big( X_{ \lceil r \rceil_h , T}^{Y_r} \big) X_{ \lceil r \rceil_h , T}^{2, Y_r} \Big)
(e_l^{(d)}, e_j^{(d)})
\,dr.
\nonumber
\end{align}
Next we want to let $\nicefrac{T}{\N}\ni h\to 0$ in~\eqref{fastfertig} in a suitable sense and first justify this.
H\"{o}lder's inequality,
inequalities~\eqref{443},~\eqref{inthemoodforlove1},~\eqref{442b},
and the fact that
$A \in \mathcal L^p( \lambda_{[0, T]} \otimes \mathbb P; \R^d)$
imply that 
\begin{equation} \label{243}
\begin{split}
\sup_{h \in \nicefrac{T}{ \N }}
&
\Bigg \| 
f' \Bigg( X_{ \lceil \mycdot \rceil_h, T}^{X_{ \lfloor \mycdot \rfloor_h, \mycdot}^{Y_{ \lfloor \mycdot \rfloor_h}}} \Bigg) 
X_{ \lceil \mycdot \rceil_h, T}^{1, X_{ \lfloor \mycdot \rfloor_h, \mycdot}^{Y_{ \lfloor \mycdot \rfloor_h }}} 
\mu \Big( \cdot, X_{ \lfloor \mycdot \rfloor _h, \mycdot }^{Y_{ \lfloor \mycdot \rfloor _h }} \Big) 
- 
f' \big(X_{ \lceil \mycdot \rceil_h , T}^{Y_{\mycdot}} \big)
X_{ \lceil \mycdot \rceil_h , T}^{1, Y_{\mycdot}} A_{\mycdot} 
\Bigg \|_{L^2( \lambda_{[0, T]} \otimes \mathbb P; \R^k)}
\\&
\leq
T^{ \! \frac{p - 2}{2p}}
\Big(
\sup_{ \substack{r, s, t \in [0, T] \\ r \leq s \leq t}} 
\Big \| 
f' \Big(X_{t, T}^{X_{r, s}^{Y_r}} \Big) 
\Big \|_{L^{ \! \frac{p}{q}}( \mathbb P; L( \R^d, \R^k))} 
\Big)
\Big( 
\sup_{ \substack{r, s, t \in [0, T] \\ r \leq s \leq t}} 
\Big \|
X_{t, T}^{1, X_{r, s}^{Y_r}} 
\Big \|_{L^{ \! \frac{2p}{p - 2(q + 1)}}( \mathbb P; L( \R^d, \R^d))} 
\Big)
\\&
\quad \quad \quad \quad \quad \quad 
\cdot
\Big(
\sup_{h \in \nicefrac{T}{ \N }}
\Big \|
\mu \big( \cdot, X_{ \lfloor \mycdot \rfloor _h, \mycdot }^{Y_{ \lfloor \mycdot \rfloor _h }} \big)
\Big \|_{L^p( \lambda_{[0, T]} \otimes \mathbb P; \R^d)}
+
\|
A
\|_{L^p( \lambda_{[0, T]} \otimes \mathbb P; \R^d)} 
\Big)
<
\infty.
\end{split}
\end{equation}
%
H\"older's inequality and inequalities~\eqref{443} and~\eqref{inthemoodforlove1} imply that for all $l, j \in \{1, \dots, d \}$ it holds that
\begin{equation} \label{1551}
\begin{split}
&
\sup_{h \in \nicefrac{T}{ \N }}
\Bigg \|
\Bigg( f'' \Bigg( X_{ \lceil \mycdot \rceil_h, T}^{X_{ \lfloor \mycdot \rfloor_h, \mycdot}^{Y_{ \lfloor \mycdot \rfloor_h}}} \Bigg) 
\Bigg( X_{ \lceil \mycdot \rceil_h, T}^{1, X_{ \lfloor \mycdot \rfloor_h, \mycdot}^{Y_{ \lfloor \mycdot \rfloor_h}}}, X_{ \lceil \mycdot \rceil_h, T}^{1, X_{ \lfloor \mycdot \rfloor_h, \mycdot}^{Y_{ \lfloor \mycdot \rfloor_h}}} \Bigg)
+ 
f' \Bigg( X_{ \lceil \mycdot \rceil_h, T}^{X_{ \lfloor \mycdot \rfloor_h, \mycdot}^{Y_{ \lfloor \mycdot \rfloor_h}}} \Bigg) 
X_{ \lceil \mycdot \rceil_h, T}^{2, X_{ \lfloor \mycdot \rfloor_h, \mycdot}^{Y_{ \lfloor \mycdot \rfloor_h}}}
\Bigg) 
\big( e^{(d)}_l, e^{(d)}_j \big)
\Bigg \|_{L^{ \! \frac{2p}{p - 4}}( \lambda_{[0, T]} \otimes \mathbb P; \R^k)}
\\&
\leq
\sup_{h \in \nicefrac{T}{ \N }}
\Bigg \|
\Bigg \|
f'' \Bigg( X_{ \lceil \mycdot \rceil_h, T}^{X_{ \lfloor \mycdot \rfloor_h, \mycdot}^{Y_{ \lfloor \mycdot \rfloor_h}}} \Bigg)
\Bigg \|_{L^{(2)}( \R^d, \R^k)}
\Bigg \| 
X_{ \lceil \mycdot \rceil_h, T}^{1, X_{ \lfloor \mycdot \rfloor_h, \mycdot}^{Y_{ \lfloor \mycdot \rfloor_h}}}
\Bigg \|_{L( \R^d, \R^d)}^2
\\&\qquad\qquad
+ 
\Bigg \|
f' \Big( X_{ \lceil \mycdot \rceil_h, T}^{X_{ \lfloor \mycdot \rfloor_h, \mycdot}^{Y_{ \lfloor \mycdot \rfloor_h}}} \Big) 
\Bigg \|_{L( \R^d, \R^k)}
\Bigg \|
X_{ \lceil \mycdot \rceil_h, T}^{2, X_{ \lfloor \mycdot \rfloor_h, \mycdot}^{Y_{ \lfloor \mycdot \rfloor_h}}}
\Bigg \|_{L^{(2)}( \R^d, \R^d)}
\Bigg \|_{L^{ \! \frac{2p}{p - 4}}( \lambda_{[0, T]} \otimes \mathbb P; \R)}
\\&
\leq
\sup_{h \in \nicefrac{T}{ \N }}
\Bigg(
\Bigg \|
f'' \Bigg( X_{ \lceil \mycdot \rceil_h, T}^{X_{ \lfloor \mycdot \rfloor_h, \mycdot}^{Y_{ \lfloor \mycdot \rfloor_h}}} \Bigg)
\Bigg \|_{L^{ \! \frac{p}{q}}( \lambda_{[0, T]} \otimes \mathbb P; L^{(2)}( \R^d, \R^k))}
\Bigg \|
X_{ \lceil \mycdot \rceil_h, T}^{1, X_{ \lfloor \mycdot \rfloor_h, \mycdot}^{Y_{ \lfloor \mycdot \rfloor_h}}}
\Bigg \|_{L^{ \! \frac{4p}{p - 2(q + 2)}}( \lambda_{[0, T]} \otimes \mathbb P; L( \R^d, \R^d))}^2
\\&
\quad \quad \quad \quad 
+ 
\Bigg \|
f' \Bigg( X_{ \lceil \mycdot \rceil_h, T}^{X_{ \lfloor \mycdot \rfloor_h, \mycdot}^{Y_{ \lfloor \mycdot \rfloor_h}}} \Bigg) 
\Bigg \|_{L^{ \! \frac{p}{q}}( \lambda_{[0, T]} \otimes \mathbb P; L( \R^d, \R^k))}
\Bigg \|
X_{ \lceil \mycdot \rceil_h, T}^{2, X_{ \lfloor \mycdot \rfloor_h, \mycdot}^{Y_{ \lfloor \mycdot \rfloor_h}}}
\Bigg \|_{L^{ \! \frac{2p}{p - 2(q + 2)}}( \lambda_{[0, T]} \otimes \mathbb P; L^{(2)}( \R^d, \R^d))}
\Bigg)
\\&
\leq
T^{ \! \frac{p - 4}{2p}}
\Bigg(
\bigg(
\sup_{ \substack{r, s, t \in [0, T] \\ r \leq s \leq t}} 
\Big \| 
f'' \Big(X_{t, T}^{X_{r, s}^{Y_r}} \Big) 
\Big \|_{L^{ \! \frac{p}{q}}( \mathbb P; L^{(2)}( \R^d, \R^k))} 
\bigg)
\bigg( 
\sup_{ \substack{r, s, t \in [0, T] \\ r \leq s \leq t}} 
\Big \|
X_{t, T}^{1, X_{r, s}^{Y_r}} 
\Big \|_{L^{ \! \frac{4p}{p - 2(q + 2)}}( \mathbb P; L( \R^d, \R^d))}^2 
\bigg)
\\&
\quad \quad \quad \quad 
+ 
\bigg(
\sup_{ \substack{r, s, t \in [0, T] \\ r \leq s \leq t}} 
\Big \| 
f' \Big(X_{t, T}^{X_{r, s}^{Y_r}} \Big) 
\Big \|_{L^{ \! \frac{p}{q}}( \mathbb P; L( \R^d, \R^k))} 
\bigg)
\bigg( 
\sup_{ \substack{r, s, t \in [0, T] \\ r \leq s \leq t}} 
\Big \|
X_{t, T}^{2, X_{r, s}^{Y_r}} 
\Big \|_{L^{ \! \frac{2p}{p - 2(q + 2)}}( \mathbb P; L^{(2)}( \R^d, \R^d))}
\bigg)
\Bigg)
< 
\infty
\end{split}
\end{equation}
and, analogously, that for all $i,j\in\{1,\ldots,d\}$ it holds that
\begin{equation} \label{1552}
\sup_{h \in \nicefrac{T}{ \N }}
\Big \|
\Big( f''(X_{ \lceil \mycdot \rceil_h , T}^{Y_{\mycdot}}) \big( X_{ \lceil \mycdot \rceil_h , T}^{1, Y_{\mycdot}}, X_{ \lceil \mycdot \rceil_h , T}^{1, Y_{\mycdot}} \big)
+
f'(X_{ \lceil \mycdot \rceil_h , T}^{Y_{\mycdot}}) X_{ \lceil \mycdot \rceil_h , T}^{2, Y_{\mycdot}} \Big)
\big( e^{(d)}_i, e^{(d)}_j \big)
\Big \|_{ L^ \frac{2p}{p - 4} ( \lambda_{[0, T]} \otimes \mathbb P; \R^k)}
<
\infty.
\end{equation}
The fact that for all $C \in \R^{d \times m}$ it holds that 
$ \sum_{i, j =1}^d |(C C^*)_{i, j}| \leq d \| C \|_{ \operatorname{HS}( \R^m, \R^d)}^2$, 
H\"older's inequality, 
assumption~\eqref{442b},
and inequality~\eqref{1551}
imply that 
\begin{equation}  \begin{split} \label{theconvergence64}
& 
\sup_{h \in \nicefrac{T}{ \N}}
\tfrac{1}{2} 
\Bigg \|
\sum_{l, j = 1}^d 
\bigg( \sigma \Big( \cdot, X_{ \lfloor \mycdot \rfloor_h, \mycdot}^{Y_{ \lfloor \mycdot \rfloor_h }} \Big) 
\Big[\sigma \Big( \cdot, X_{ \lfloor \mycdot \rfloor_h , \mycdot}^{Y_{ \lfloor \mycdot \rfloor_h}} \Big)\Big]^* \bigg)_{l, j}
\\&
\quad
\cdot
\Bigg(
f'' \Bigg( X_{ \lceil \mycdot \rceil_h, T}^{X_{ \lfloor \mycdot \rfloor_h, \mycdot}^{Y_{ \lfloor \mycdot \rfloor_h}}} \Bigg) 
\Bigg( X_{ \lceil \mycdot \rceil_h, T}^{1, X_{ \lfloor \mycdot \rfloor_h, \mycdot}^{Y_{ \lfloor \mycdot \rfloor_h}}}, X_{ \lceil \mycdot \rceil_h, T}^{1, X_{ \lfloor \mycdot \rfloor_h, \mycdot}^{Y_{ \lfloor \mycdot \rfloor_h}}} \Bigg)
+ f' \Bigg( X_{ \lceil \mycdot \rceil_h, T}^{X_{ \lfloor \mycdot \rfloor_h, \mycdot}^{Y_{ \lfloor \mycdot \rfloor_h}}} \Bigg) 
X_{ \lceil \mycdot \rceil_h, T}^{2, X_{ \lfloor \mycdot \rfloor_h, \mycdot}^{Y_{ \lfloor \mycdot \rfloor_h}}}
\Bigg)
\big( e^{(d)}_l, e^{(d)}_j \big)
\Bigg \|_{L^2( \lambda_{[0, T]} \otimes \mathbb P; \R^k)}
\\&\leq
\sup_{h \in \nicefrac{T}{ \N}}
\tfrac{1}{2} 
\Bigg \|
d\left\|\sigma\Big( \cdot, X_{ \lfloor \mycdot \rfloor_h, \mycdot}^{Y_{ \lfloor \mycdot \rfloor_h }} \Big)\right\|_{\operatorname{HS}( \R^m, \R^d))}^2
\\&
\quad
\cdot
\sum_{l, j = 1}^d 
\Bigg|
f'' \Bigg( X_{ \lceil \mycdot \rceil_h, T}^{X_{ \lfloor \mycdot \rfloor_h, \mycdot}^{Y_{ \lfloor \mycdot \rfloor_h}}} \Bigg) 
\Bigg( X_{ \lceil \mycdot \rceil_h, T}^{1, X_{ \lfloor \mycdot \rfloor_h, \mycdot}^{Y_{ \lfloor \mycdot \rfloor_h}}}, X_{ \lceil \mycdot \rceil_h, T}^{1, X_{ \lfloor \mycdot \rfloor_h, \mycdot}^{Y_{ \lfloor \mycdot \rfloor_h}}} \Bigg)
+ f' \Bigg( X_{ \lceil \mycdot \rceil_h, T}^{X_{ \lfloor \mycdot \rfloor_h, \mycdot}^{Y_{ \lfloor \mycdot \rfloor_h}}} \Bigg) 
X_{ \lceil \mycdot \rceil_h, T}^{2, X_{ \lfloor \mycdot \rfloor_h, \mycdot}^{Y_{ \lfloor \mycdot \rfloor_h}}}
\Bigg|
\big( e^{(d)}_l, e^{(d)}_j \big)
\Bigg \|_{L^2( \lambda_{[0, T]} \otimes \mathbb P; \R^k)}
\\&
\leq
\sup_{h \in \nicefrac{T}{ \N}}
\tfrac{d}{2} 
\Bigg(
\Big \|
\sigma \Big( \cdot, X_{ \lfloor \mycdot \rfloor _h, \mycdot }^{Y_{ \lfloor \mycdot \rfloor _h }} \Big)
\Big \|_{L^p( \lambda_{[0, T]} \otimes \mathbb P; \R^{d\times m})}^2
\\&
\quad
\cdot
\sum_{l,j \in \{1, \dots, d \}}
\Bigg \|
f'' \Bigg( X_{ \lceil \mycdot \rceil_h, T}^{X_{ \lfloor \mycdot \rfloor_h, \mycdot}^{Y_{ \lfloor \mycdot \rfloor_h}}} \Bigg) 
\Bigg(X_{ \lceil \mycdot \rceil_h, T}^{1, X_{ \lfloor \mycdot \rfloor_h, \mycdot}^{Y_{ \lfloor \mycdot \rfloor_h}}}, X_{ \lceil \mycdot \rceil_h, T}^{1, X_{ \lfloor \mycdot \rfloor_h, \mycdot}^{Y_{ \lfloor \mycdot \rfloor_h}}} \Bigg)
+ f' \Bigg( X_{ \lceil \mycdot \rceil_h, T}^{X_{ \lfloor \mycdot \rfloor_h, \mycdot}^{Y_{ \lfloor \mycdot \rfloor_h}}} \Bigg) 
X_{ \lceil \mycdot \rceil_h, T}^{2, X_{ \lfloor \mycdot \rfloor_h, \mycdot}^{Y_{ \lfloor \mycdot \rfloor_h}}}
\Bigg)
\big( e^{(d)}_l, e^{(d)}_j \big)
\Bigg \|_{L^{ \! \frac{2p}{p - 4}}( \lambda_{[0, T]} \otimes \mathbb P; \R^k)}
\Bigg)
\\&<\infty.
\end{split}     \end{equation}
Analogously, the fact that for all $C \in \R^{d \times m}$ 
it holds that 
$ \sum_{i, j =1}^d |(C C^*)_{i, j}| \leq d \| C \|_{ \operatorname{HS}( \R^m, \R^d)}^2$,
H\"older's inequality,
the assumption $B \in \mathcal L^p( \lambda_{[0, T]} \otimes \mathbb P; \R^{d\times m})$,
and
inequality~\eqref{1552}
yield that 
\begin{equation} \label{theconvergence65}
\begin{split}
& 
\sup_{h \in \nicefrac{T}{ \N}}
\tfrac{1}{2} 
\bigg \|
\sum_{l, j = 1}^d
\left(B_{\mycdot} [B_{\mycdot}]^* \right)_{l, j}
\Big( f''(X_{ \lceil \mycdot \rceil_h , T}^{Y_{\mycdot}}) \big( X_{ \lceil \mycdot \rceil_h , T}^{1, Y_{\mycdot}}, X_{ \lceil \mycdot \rceil_h , T}^{1, Y_{\mycdot}} \big)
+
f'(X_{ \lceil \mycdot \rceil_h , T}^{Y_{\mycdot}}) X_{ \lceil \mycdot \rceil_h , T}^{2, Y_{\mycdot}} \Big)
\big( e^{(d)}_l, e^{(d)}_j \big)
\bigg \|_{L^2( \lambda_{[0, T]} \otimes \mathbb P; \R^k)}
\\&
\leq
\frac{d}{2}
\|
B
\|_{L^p( \lambda_{[0, T]} \otimes \mathbb P; \R^{d\times m})}^2
\\&
\quad 
\cdot
\sum_{l,j \in \{1, \dots, d \}}
\sup_{h \in \nicefrac{T}{ \N}}
\bigg \|
\Big( f''(X_{ \lceil \mycdot \rceil_h , T}^{Y_{\mycdot}}) \big( X_{ \lceil \mycdot \rceil_h , T}^{1, Y_{\mycdot}}, X_{ \lceil \mycdot \rceil_h , T}^{1, Y_{\mycdot}} \big)
+
f'(X_{ \lceil \mycdot \rceil_h , T}^{Y_{\mycdot}}) X_{ \lceil \mycdot \rceil_h , T}^{2, Y_{\mycdot}} \Big)
\big( e^{(d)}_l, e^{(d)}_j \big)
\bigg \|_{L^{ \! \frac{2p}{p - 4}}( \lambda_{[0, T]} \otimes \mathbb P; \R^k)}
< \infty.
\end{split}
\end{equation}
Next Klenke~\cite[Corollary 6.21 and  Theorem 6.25]{Klenke2008} together with the uniform $L^2$-bounds in~\eqref{243},~\eqref{theconvergence64}, and~\eqref{theconvergence65},
continuity of $f'$ and of $f''$,
path continuity of $Y$ and of 
$ \Delta_T \times O \ni (s, t, x) \mapsto X_{s, t}^x \in O$, 
and 
$ \inf_{r \in [0, T]} \mathbb P(X_{r, r}^{Y_r} = Y_r) = 1$
imply that 
\begin{equation}  \begin{split} \label{gertr}
&\lim_{ \nicefrac{T}{ \N} \ni h \searrow 0}
\Bigg \|
\smallint_0^T f' \Bigg(X_{ \lceil r \rceil_h, T}^{X_{ \lfloor r \rfloor_h, r}^{Y_{ \lfloor r \rfloor_h}}} \Bigg) 
X_{ \lceil r \rceil_h, T}^{1, X_{ \lfloor r \rfloor_h, r}^{Y_{ \lfloor r \rfloor_h }}} \mu(r, X_{\lfloor r\rfloor_h,r}^{Y_{\lfloor r\rfloor_h}}) 
- 
f' \big(X_{ \lceil r \rceil_h , T}^{Y_r} \big)
X_{ \lceil r \rceil_h , T}^{1, Y_r} A_r \,dr
\\&
-
\smallint_0^T
f' \big(X_{r, T}^{Y_r} \big)
X_{r, T}^{1, Y_r}
\Big( \mu \Big(r, Y_r \Big) - A_r \Big) \,dr
\\&
+ 
\tfrac{1}{2} 
\sum_{l, j = 1}^d 
\smallint_0^T
\Big( \sigma \big(r, X_{ \lfloor r \rfloor_h, r}^{Y_{ \lfloor r \rfloor_h }} \big) \big[\sigma \big(r, X_{ \lfloor r \rfloor_h , r}^{Y_{ \lfloor r \rfloor_h}} \big)\big]^* \Big)_{l, j}
\\&\qquad\cdot
\Bigg( 
f'' \Bigg( X_{ \lceil r \rceil_h, T}^{X_{ \lfloor r \rfloor_h, r}^{Y_{ \lfloor r \rfloor_h}}} \Bigg) 
\Bigg( X_{ \lceil r \rceil_h, T}^{1, X_{ \lfloor r \rfloor_h, r}^{Y_{ \lfloor r \rfloor_h}}}, X_{ \lceil r \rceil_h, T}^{1, X_{ \lfloor r \rfloor_h, r}^{Y_{ \lfloor r \rfloor_h}}} \Bigg)
+ 
f' \Bigg( X_{ \lceil r \rceil_h, T}^{X_{ \lfloor r \rfloor_h, r}^{Y_{ \lfloor r \rfloor_h}}} \Bigg) 
X_{ \lceil r \rceil_h, T}^{2, X_{ \lfloor r \rfloor_h, r}^{Y_{ \lfloor r \rfloor_h}}}
\Bigg)
\big( e^{(d)}_l, e^{(d)}_j \big)
\,dr
\\&
- 
\tfrac{1}{2}
\sum_{l, j = 1}^d
\smallint_0^T
\left(B_r [B_r]^* \right)_{l, j}
\Big( f'' \big( X_{ \lceil r \rceil_h , T}^{Y_r} \big) \big( X_{ \lceil r \rceil_h , T}^{1, Y_r}, X_{ \lceil r \rceil_h , T}^{1, Y_r} \big)
+
f' \big( X_{ \lceil r \rceil_h , T}^{Y_r} \big) X_{ \lceil r \rceil_h , T}^{2, Y_r} \Big)
(e_l^{(d)}, e_j^{(d)})
\,dr
\\& 
- 
\tfrac{1}{2} 
\sum_{l, j=1}^d
\smallint_0^T
\left( \sigma(r, Y_r) [\sigma(r, Y_r)]^*
- 
B_r [B_r]^*
\right)_{l, j}
\Big(
f'' \big(X_{r, T}^{Y_r} \big)
\big(X_{r, T}^{1, Y_r}, X_{r, T}^{1, Y_r} \big)
+ 
f' \big(X_{r, T}^{Y_r} \big)
X_{r, T}^{2, Y_r}
\Big)
(e_l^{(d)}, e_j^{(d)})
\,dr
\Bigg \|_{L^1( \mathbb P; \R^k)} 
\\&
=0.
\end{split}     \end{equation}
Inequality~\eqref{442} implies that for all $x, y \in O$ it holds that 
\begin{equation} \label{fxbetrag}
\| f(x) - f(y) \|_{ \R^k}
\leq
\|f(x) \|_{ \R^k}
+
\| f(y) \|_{ \R^k}
\leq
c
(1 + \| x \|_{ \R^d})
(1 + \| x \|_{ \R^d})^q
+
c
(1 + \| y \|_{ \R^d})
(1 + \| y \|_{ \R^d})^q.
\end{equation}
This,
H\"older's inequality, 
the fact that $2q + 2<p$,
the fact that
$\mathbb{P}\Big(X_{0,T}^{Y_0}=X_{0,T}^{X_{0,0}^{Y_0}}\Big)=1
=\mathbb{P}\Big(Y_T=X_{T,T}^{X_{T,T}^{Y_T}}\Big)$,
and
inequality~\eqref{inthemoodforlove1}
show that 
\begin{align} \label{nochmehrquarbroetchen}
\nonumber 
&
\big \|
f \big(X_{0, T}^{Y_0} \big) - f \big(Y_T \big) 
\big \|_{L^2( \mathbb P; \R^k)}
\\&
\leq
c
\big \| 
\big ( 1 + \big \| X_{0, T}^{Y_0} \big \|_{ \R^d} \big )^{1+q}
\big \|_{L^2( \mathbb P; \R)}
+
c
\big \| 
\big( 1 + \| Y_T \|_{ \R^d} \big)^{1+q}
\big \|_{L^2( \mathbb P; \R)}
\\&
\nonumber
\leq
c
\big (
1 + \big \| X_{0, T}^{Y_0} \big \|_{L^{2q + 2}( \mathbb P; \R^d)}
\big )^{q + 1}
+ 
c
\big(
1 + \| Y_T \|_{L^{2q + 2}( \mathbb P; \R^d)}
\big)^{q + 1}
\\&
\leq
\sup_{ \substack{r, s, t \in [0, T] \\ r \leq s \leq t}}
2c
\Big(
1 + \left\| X_{t,T}^{X_{r,s}^{Y_r}} \right\|_{L^{p}( \mathbb P; \R^d)}
\Big)^{q + 1}
< \infty. 
\nonumber
\end{align}
Equation~\eqref{fastfertig} and inequalities~\eqref{nochmehrquarbroetchen},~\eqref{243},~\eqref{theconvergence64}, and~\eqref{theconvergence65} imply that there exists a constant $K \in [0, \infty)$ such that for all $h \in \nicefrac{T}{\N}$ it holds that 
\begin{equation} \label{quarkbroetchen}
\begin{split}
& 
\Bigg \|
\smallint_0^T 
f' \Bigg(X_{ \lceil r \rceil_h, T}^{X_{ \lfloor r \rfloor_h, r}^{Y_{ \lfloor r \rfloor_h}}} \Bigg) 
X_{ \lceil r \rceil_h, T}^{1, X_{ \lfloor r \rfloor_h, r}^{Y_{ \lfloor r \rfloor_h }}} 
\sigma \Big(r, X_{ \lfloor r \rfloor_h , r}^{Y_{ \lfloor r \rfloor_h}} \Big) 
- f' \Big( X_{ \lceil r \rceil_h, T}^{Y_r} \Big)
X_{ \lceil r \rceil_h , T}^{1, Y_r} B_r 
\, \delta
W_r^{\mathbb{F}_0} 
\Bigg \|_{L^2( \mathbb P; \R^k)}
\\&
\leq
\big \|
f \big(X_{0, T}^{Y_0} \big) - f \left(Y_T \right) 
\big \|_{L^2( \mathbb P; \R^k)}
+ 
\Big \|
\smallint_0^T f' \bigg(X_{ \lceil r \rceil_h, T}^{X_{ \lfloor r \rfloor_h, r}^{Y_{ \lfloor r \rfloor_h}}} \bigg) X_{ \lceil r \rceil_h, T}^{1, X_{ \lfloor r \rfloor_h, r}^{Y_{ \lfloor r \rfloor_h }}} \mu(r, X_{ \lfloor r \rfloor_h , r}^{Y_{ \lfloor r \rfloor_h}} ) 
- f' \big(X_{ \lceil r \rceil_h , T}^{Y_r} \big)
X_{ \lceil r \rceil_h, T}^{1, Y_r} A_r \,dr
\Big \|_{L^2( \mathbb P; \R^k)}
\\&
+ 
\Bigg \|
\tfrac{1}{2} \sum_{l, j=1}^d \smallint_0^T
\Big( \sigma \big(r, X_{ \lfloor r \rfloor_h, r}^{Y_{ \lfloor r \rfloor_h}} \big) \big[\sigma \big(r, X_{ \lfloor r \rfloor_h, r}^{Y_{ \lfloor r \rfloor_h}} \big)\big]^* \Big)_{l, j} 
\\& 
\quad \quad \quad \quad 
\cdot 
\Bigg(
f'' \Bigg(X_{ \lceil r \rceil_h, T}^{X_{ \lfloor r \rfloor_h , r}^{Y_{ \lfloor r \rfloor_h}}} \Bigg) 
\Bigg( X_{ \lceil r \rceil_h, T}^{1, X_{ \lfloor r \rfloor_h, r}^{Y_{ \lfloor r \rfloor_h}}}, X_{ \lceil r \rceil_h, T}^{1, X_{ \lfloor r \rfloor_h, r}^{Y_{ \lfloor r \rfloor_h}}} \Bigg) 
+ 
f' \Bigg( X_{ \lceil r \rceil_h, T}^{X_{ \lfloor r \rfloor_h, r}^{Y_{ \lfloor r \rfloor_h}}} \Bigg)
X_{ \lceil r \rceil_h, T}^{2, X_{ \lfloor r \rfloor_h, r}^{Y_{ \lfloor r \rfloor_h}}} 
\Bigg)
(e_l^{(d)}, e_j^{(d)})
\\&
\quad \quad \quad \quad 
- 
\big(B_r [B_r]^* \big)_{l, j}
\Bigg(
f'' \big(X_{ \lceil r \rceil_h , T}^{Y_r} \big)
\big(
X_{ \lceil r \rceil_h , T}^{1, Y_r}, 
X_{ \lceil r \rceil_h , T}^{1, Y_r}
\big)
+
f' \big(X_{ \lceil r \rceil_h , T}^{Y_r} \big)
X_{ \lceil r \rceil_h , T}^{2, Y_r}
\Bigg)
(e_l^{(d)}, e_j^{(d)})
\,dr
\Bigg \|_{L^2( \mathbb P; \R^k)}
< K.
\end{split}
\end{equation}
The fact that $Y$, $X$, $X^1$ are continuous random fields,
continuity of $f'$, and
the fact that $ \inf_{r \in [0, T]} \mathbb P ( X_{r, r}^{Y_r} = Y_r) = 1$
yield that for all $r\in[0,T]$ it holds $\mathbb{P}$-a.s.\ that
\begin{equation}  \begin{split}\label{eq:conv.prob}
&\lim_{ \nicefrac{T}{ \N} \ni h \searrow 0}
\left(f' \Bigg(X_{ \lceil r \rceil_h, T}^{X_{ \lfloor r \rfloor_h, r }^{Y_{ \lfloor r \rfloor_h}}} \Bigg) 
X_{ \lceil r \rceil_h, T}^{1, X_{ \lfloor r \rfloor_h, r }^{Y_{ \lfloor r \rfloor_h }}} 
\sigma \Big( r, X_{ \lfloor r \rfloor_h, r}^{Y_{ \lfloor r \rfloor_h }} \Big) 
- 
f' \Big( X_{ \lceil r \rceil_h, T}^{Y_ r } \Big)
X_{ \lceil r \rceil_h , T}^{1, Y_ r } B_ r 
\right)
\\&
=
f' \big(X_{ r , T}^{Y_ r } \big)X_{ r , T}^{1, Y r }
\Big( \sigma( r , Y_ r ) - B_ r \Big).
\end{split}     \end{equation}
This,
Fatou's lemma,
and
the inequalities~\eqref{245} and~\eqref{238}
yield that the sequence
\begin{equation}  \begin{split}\label{eq:sequence}
\left(
f' \Bigg(X_{ \lceil \mycdot \rceil_h, T}^{X_{ \lfloor \mycdot \rfloor_h, \mycdot }^{Y_{ \lfloor \mycdot \rfloor_h}}} \Bigg) 
X_{ \lceil \mycdot \rceil_h, T}^{1, X_{ \lfloor \mycdot \rfloor_h, \mycdot }^{Y_{ \lfloor \mycdot \rfloor_h }}} 
\sigma \Big( \cdot, X_{ \lfloor \mycdot \rfloor_h, \mycdot}^{Y_{ \lfloor \mycdot \rfloor_h }} \Big) 
- 
f' \Big( X_{ \lceil \mycdot \rceil_h, T}^{Y_{\mycdot} } \Big)
X_{ \lceil \mycdot \rceil_h , T}^{1, Y_{\mycdot} } B_{\mycdot} 
-
f' \big(X_{ \mycdot , T}^{Y_{\mycdot} } \big)X_{ \mycdot , T}^{1, Y_{\mycdot} }
\Big( \sigma( \cdot , Y_{\mycdot} ) - B_{\mycdot} \Big)
\right)_{ h\in \nicefrac{T}{ \N} }
\end{split}     \end{equation}
is bounded in $L^2(\lambda_{[0,T]}\otimes\mathbb{P};\R^{k\times m})$.
This, the fact that every bounded sequence in the separable Hilbert space
$L^2(\lambda_{[0,T]}\otimes\mathbb{P};\R^{k\times m})$
has a weakly converging subsequence (e.g., Kato~\cite[Lemma 5.1.4]{Kato1980}),
and the convergence~\eqref{eq:conv.prob} ensure that
the sequence~\eqref{eq:sequence} converges to $0$ in the weak topology of 
$ L^2( \lambda_{[0, T]} \otimes \mathbb P; \R^{k\times m})$
as $ \nicefrac{T}{ \N } \ni h \searrow 0$.
This, the fact that the processes
\begin{equation}  \begin{split}
\left(  f' \Big(X_{ \lceil r \rceil_h, T}^{X_{ \lfloor r \rfloor_h, r}^{Y_{ \lfloor r \rfloor_h}}} \Big) 
X_{ \lceil r \rceil_h, T}^{1, X_{ \lfloor r \rfloor_h, r}^{Y_{ \lfloor r \rfloor_h }}} 
\sigma \Big( r, X_{ \lfloor r \rfloor _h, r}^{Y_{ \lfloor r \rfloor _h }} \Big)
-
f' \big(X_{\lceil r\rceil_h, T}^{Y_r} \big) X_{\lceil r\rceil_h, T}^{1, Y_r} B_r
\right)_{r\in[0,T]},
\quad h\in T/\N,
\end{split}     \end{equation}
are
Skorohod-integrable,~\eqref{quarkbroetchen},
and Lemma~\ref{l:Skorohod.integral.closed} imply that the stochastic process
\begin{equation}  \begin{split}
\big(
f' \big(X_{r, T}^{Y_r} \big)X_{r, T}^{1, Y_r}
( \sigma(r, Y_r) - B_r)
\big)_{r \in [0, T]}
\end{split}     \end{equation}
is
Skorohod-integrable 
and that  for every $\mathbb{F}_T$/$\mathcal{B}([-1,1]^k)$-measurable
function $Z\colon \Omega\to[-1,1]^k$
it holds that
\begin{equation}  \begin{split} \label{drittekonvergenz}
&\lim_{\nicefrac{T}{\N}\ni h\searrow 0}
\E\Bigg[\Bigg\langle Z,
\smallint_0^T 
f' \Bigg( X_{ \lceil r \rceil_h, T}^{X_{ \lfloor r \rfloor_h, r}^{Y_{ \lfloor r \rfloor_h}}} \Bigg) 
X_{ \lceil r \rceil_h, T}^{1, X_{ \lfloor r \rfloor_h, r}^{Y_{ \lfloor r \rfloor_h }}} 
\sigma \Big(r, X_{ \lfloor r \rfloor_h, r}^{Y_{ \lfloor r \rfloor_h }} \Big) 
- 
f' \Big( X_{ \lceil r \rceil_h, T}^{Y_r} \Big)
X_{ \lceil r \rceil_h , T}^{1, Y_r} B_r
\,\delta
W_r^{\mathbb{F}_0}
\\&\qquad\qquad\qquad\quad
-\smallint_0^T f' \big( X_{r, T}^{Y_r} \big) X_{r, T}^{1, Y_r} \sigma(r, Y_r) - f'\big( X_{r, T}^{Y_r} \big) X_{r, T}^{1, Y_r} B_r
\,\delta
W_r^{\mathbb{F}_0}
\Bigg\rangle_{\R^k}\Bigg]=0.
\end{split}     \end{equation}
Equation~\eqref{fastfertig} and the convergences~\eqref{gertr} and~\eqref{drittekonvergenz} imply that
for every $\mathbb{F}_T$/$\mathcal{B}([-1,1]^k)$-measurable
function $Z\colon \Omega\to[-1,1]^k$
it holds that
\begin{equation}  \begin{split}
&\E\Bigg[\Bigg\langle Z,
\smallint_0^T
f' \big(X_{r, T}^{Y_r} \big)
X_{r, T}^{1, Y_r}
\Big( \mu(r, Y_r) - A_r \Big) \,dr
+ 
\smallint_0^T
f' \big(X_{r, T}^{Y_r} \big)
X_{r, T}^{1, Y_r}
\Big( \sigma(r, Y_r) - B_r \Big) 
\,\delta
W_r^{\mathbb{F}_0}
\\
&\quad 
+ 
\tfrac{1}{2} 
\sum_{l, j=1}^d
\smallint_0^T
\Big( \sigma(r, Y_r) [ \sigma(r, Y_r)]^*
- 
B_r [B_r]^*
\Big)_{l, j}
\Big(
f'' \big(X_{r, T}^{Y_r} \big)
\big(X_{r, T}^{1, Y_r}, X_{r, T}^{1, Y_r} \big)
+ 
f' \big(X_{r, T}^{Y_r} \big)
X_{r, T}^{2, Y_r}
\Big)
\big( e_l^{(d)}, e_j^{(d)} \big)
\,dr
\\
&\quad 
-f \big(X_{0, T}^{Y_0} \big)
+ 
f (Y_T )
\Bigg\rangle_{\R^k}\Bigg]=0.
\end{split}     \end{equation}
This implies equation~\eqref{eq:l.perturbation.formula}.
The proof of Theorem~\ref{thm:perturbation.formula} is thus completed.
\end{proof}

\section*{Appendix: The Skorohod integral with respect to Brownian motion and additional independent information}
\stepcounter{section}
\setcounter{section}{1}
\renewcommand{\thesection}{\Alph{section}}%

In this appendix we introduce the Skorohod integral with respect to a Brownian motion $W$
and an additional sigma-algebra $ \mathbb{F}_0$ which is independent of $W$. As a motivation, note that for
every  probability space $( \Omega, \mathcal{F}, \mathbb P)$
and
every standard Brownian motion
$W \colon [0, 3] \times \Omega \to \R$ 
the It\^{o} integrals
$ \smallint_0^1 \sin(W_s(W_2-W_1)) \,dW_s$
and
$ \smallint_1^2\sin(W_s(W_3-W_2)) \,dW_s$
are well-defined (however with respect to different filtrations)
but their sum cannot be written as It\^{o} integral 
$ \smallint_0^2 \sin( W_s(  W_{\lceil s \rceil_1 +1}-W_{\lceil s \rceil_1 }))  \,dW_s$
(which is not well-defined as It\^o integral).
In this appendix we provide sufficient results to rewrite It\^{o} integrals as Skorohod integrals and then to write the sum of these as a single Skorohod integral. 


\newcommand{\spaceH}{\R^d}
\newcommand{\spaceU}{\R^m}
\begin{setting}\label{setting4}
Let 
$d,m\in\N$,
let 
$S,T \in \R$ satisfy $S<T$,
let
$( \Omega, \mathcal{F}, \mathbb P)$
be a probability space, 
let
$W \colon [S, T] \times \Omega \rightarrow \spaceU$
be a stochastic process such that $(W_{S+t}-W_S)_{t\in[0,T-S]}$ is a standard
Brownian motion with continuous sample paths,
let $ \mathbb{F}_S \subseteq \mathcal{F}$ be a sigma-algebra which is independent of  $\mathfrak{S}(W_{t}-W_S\colon t\in[S,T])$,
let $\mathcal{N}=\{A\in\mathcal{F}\colon \mathbb{P}(A)=0\}$,
let $ \mathbb{F}_T \subseteq \mathcal{F}$ be the sigma-algebra which satisfies that
$\mathbb{F}_T=\mathfrak{S}(\mathbb{F}_S\cup\mathfrak{S}(W_t-W_S\colon t\in[S,T])\cup\mathcal N)$,
let 
$ \mathcal S( \mathbb P, \mathbb{F}_S, W; \spaceH ) \subseteq L^2 ( \mathbb P|_{\mathbb{F}_T}; \spaceH )$
be the subset with the property that 
\begin{equation}
\mathcal S( \mathbb P, \mathbb{F}_S, W; \spaceH )
=
\begin{Bmatrix}
&F \in L^2 ( \mathbb P|_{ \mathbb{F}_T}; \spaceH ) \colon
\exists n \in \N, 
\exists \phi_1, \dots, \phi_n \in \mathcal L^2( \lambda_{[S, T]}; \spaceU ), 
\\&
\exists f \in C_b^{ \! \infty, \mathfrak{S}( \mathbb{F}_S \cup \mathcal N)}( \R^n \times \Omega, \R), 
\exists h \in \spaceH 
\textnormal{ such that it holds } \mathbb P\textnormal{-a.s.\ that }
\\&
F= 
f \big( \smallint_S^T \phi_1(r)\,dW_r, \dots, \smallint_S^T \phi_n(r)\,dW_r \big) h 
\end{Bmatrix},
\end{equation}
and for all $s, t \in [S, T]$ satisfying that $s < t$ let 
$ \mathbb{F}_{[S, s] \cup [t, T]} \subseteq \mathcal{F}$ be the sigma-algebra with the property that 
$ \mathbb{F}_{[S, s] \cup [t, T]} = \mathfrak{S}( \mathbb{F}_S \cup \mathfrak{S}(W_r-W_S \colon r \in [S, s]) \cup \mathfrak{S}(W_r - W_t \colon r \in [t, T]) \cup \mathcal N)$.
\end{setting}

\begin{definition} \label{d:Malliavin.elementary}
Assume Setting~\ref{setting4}. 
The \textbf{extended Malliavin differential} operator
\begin{equation}
\mathcal D( \mathbb P, \mathbb{F}_S, W; \spaceH )
\colon
\mathbb D^{(1, 2)}( \mathbb P, \mathbb{F}_S, W; \spaceH )
\to
L^2( \mathbb P|_{ \mathbb{F}_T}; L^2( \lambda_{[S, T]}; \R^{d\times m}))
\end{equation}
is the closed linear operator with the property that for all 
$F \in\mathcal S( \mathbb P, \mathbb{F}_S, W; \spaceH )$ 
with the property that 
$ \exists n \in \N$, 
$ \exists \phi_1, \dots, \phi_n \in \mathcal L^2( \lambda_{[S, T]}; \spaceU)$, 
$\exists f \in C_b^{ \! \infty, \mathfrak{S}( \mathbb{F}_S \cup \mathcal N)}( \R^n \times \Omega, \R)$,
$ \exists h \in \spaceH $
such that it holds $ \mathbb P$-a.s.\ that
$F= f \big( \smallint_S^T \phi_1(r)\,dW_r , \dots, \smallint_S^T \phi_n(r)\,dW_r \big) h$
it holds $ \lambda_{[S,T]}\otimes \mathbb P$-a.e.\ that
\begin{equation}\label{eq:MD}
\mathcal D( \mathbb P, \mathbb{F}_S, W; \spaceH ) F 
= \sum_{i=1}^n \frac{ \partial f }{ \partial x_{i}} 
\Big( \smallint_S^T \phi_1(s)\,dW_s , \dots, \smallint_S^T \phi_n(s)\,dW_s \Big) 
\phi_i
h
\end{equation}
and where 
$ \mathbb D^{(1, 2)}( \mathbb P, \mathbb{F}_S, W; \spaceH )$
is the closure of 
$ \operatorname{span}( \mathcal S( \mathbb P, \mathbb{F}_S, W; \spaceH )) \subseteq L^2 ( \mathbb P|_{ \mathbb{F}_T}; \spaceH )$ 
with respect to the norm 
\begin{equation}
\| \cdot \|_{ \mathbb D^{(1, 2)}( \mathbb P, \mathbb{F}_S, W; \spaceH )} 
=
\Big( \mathbb E \Big[ \| \cdot \|^2_{\spaceH} + \| \mathcal D( \mathbb P, \mathbb{F}_S, W; \spaceH ) \mycdot \|^2_{L^2( \lambda_{[S, T]}; \spaceH )} \Big] \Big)^{ \! \tfrac{1}{2}}.
\end{equation}
We write 
$
\mathcal D
=
\mathcal D( \mathbb P, \mathfrak{S}(\mathcal{N}), W; \spaceH )
$
and denote $ \mathcal D$ as the \textbf{classical Malliavin derivative}. 
\end{definition}

The following lemma, Lemma~\ref{l:MD.well-defined}, shows that
the extended Malliavin derivative is well-defined (in particular, the left-hand side of~\eqref{eq:MD} does not depend on the representative
and such a closed linear operator exists).
The proof of Lemma~\ref{l:MD.well-defined} is almost literally identical to the proofs of Proposition 4.2 and Proposition 4.4 in Kruse~\cite{Kruse2014_PhD_Thesis} and therefore omitted.
\begin{lemma}\label{l:MD.well-defined}
Assume Setting~\ref{setting4}. 
Then the operator
\begin{equation}  \begin{split}
\mathcal D( \mathbb P, \mathbb{F}_S, W; \spaceH )
\colon
\mathbb D^{(1, 2)}( \mathbb P, \mathbb{F}_S, W; \spaceH )
\to
L^2( \mathbb P|_{ \mathbb{F}_T}; L^2( \lambda_{[S, T]}; \R^{d\times m}))
\end{split}     \end{equation}
is well-defined.
\end{lemma}

The following lemma, Lemma~\ref{l:suff.rich}, shows that
the set $\mathcal S( \mathbb P, \mathbb{F}_S, W;  \spaceH ) $ is sufficiently rich.
The proof of Lemma~\ref{l:suff.rich} is standard and therefore omitted.
\begin{lemma}\label{l:suff.rich}
Assume Setting~\ref{setting4}. 
Then $\operatorname{span} \big( \mathcal S( \mathbb P, \mathbb{F}_S, W; \spaceH ) \big)$ is dense in $L^2( \mathbb P|_{ \mathbb{F}_T}; \spaceH )$.
\end{lemma}

In particular,  Lemma~\ref{l:suff.rich} implies that the extended Malliavin differential operator is densely defined.
Next we introduce the adjoint of the densely defined extended Malliavin differential operator.
\begin{definition} \label{d:integral}
Assume Setting~\ref{setting4}. 
The \textbf{extended Skorohod integral} is the linear operator 
\begin{equation}
\delta ( \mathbb P, \mathbb{F}_S, W; \spaceH )
\colon 
\operatorname{Dom}_{ \delta}( \mathbb P, \mathbb{F}_S, W; \spaceH )
\to 
L^2( \mathbb P|_{ \mathbb{F}_T}; \spaceH )
\end{equation}
which satisfies that
$X \in L^2( \mathbb P|_{ \mathbb{F}_T}; L^2( \lambda_{[S, T]}; \R^{d\times m}))$
is in the domain 
$ \operatorname{Dom}_{ \delta}( \mathbb P, \mathbb{F}_S, W; \spaceH )$
if and only if there exists a
$c \in [0, \infty)$
with the property that for all 
$F \in \operatorname{span} \big(\mathcal S ( \mathbb P, \mathbb{F}_S, W; \spaceH )\big)$ 
it holds that 
\begin{equation}
\mathbb E [ \langle \mathcal D( \mathbb P, \mathbb{F}_S, W; \spaceH ) F, X \rangle_{ L^2( \lambda_{[S, T]}; \R^{d\times m})}]
\leq 
c \|F \|_{L^2( \mathbb P; \spaceH )}
\end{equation}
and which satisfies that
for all 
$X \in \operatorname{Dom}_{ \delta}( \mathbb P, \mathbb{F}_S, W; \spaceH )$,
$F \in \mathcal S ( \mathbb P, \mathbb{F}_S, W; \spaceH )$ 
it holds that
\begin{equation} \label{eq:characterization.Skorohod}
\mathbb E \Big[ \Big \langle F,  \delta ( \mathbb P, \mathbb{F}_S, W; \spaceH )(X) \Big \rangle_{\spaceH} \Big]
= 
\mathbb E \Big[ \Big \langle \mathcal D ( \mathbb P, \mathbb{F}_S, W; \spaceH ) F, X \Big \rangle_{L^2( \lambda_{[S, T]}; \R^{d\times m})} \Big].
\end{equation}
We say that $X$ is $(\mathbb{P},\mathbb{F}_S,W;\spaceH)$\textbf{-Skorohod-integrable}
if and only if $X\in
 \operatorname{Dom}_{ \delta}( \mathbb P, \mathbb{F}_S, W; \spaceH )$.
For all $X\in \operatorname{Dom}_{ \delta}( \mathbb P, \mathbb{F}_S, W; \spaceH )$
we denote by
$\smallint_S^T X_r \,\delta W_r^{\mathbb{F}_S}$ the equivalence class satisfying that
\begin{equation}  \begin{split}
  \smallint_S^T X_r \,\delta W_r^{\mathbb{F}_S}
  =
  \delta ( \mathbb P, \mathbb{F}_S, W; \spaceH )(X).
\end{split}     \end{equation}
For all $X\in \operatorname{Dom}_{ \delta}( \mathbb P, \mathfrak{S}(\mathcal{N}), W; \spaceH )$
we denote by
$\smallint_S^T X_r \,\delta W_r$ the equivalence class satisfying that
\begin{equation}  \begin{split}
  \smallint_S^T X_r \,\delta W_r
  =
  \smallint_S^T X_r \,\delta W_r^{\mathfrak{S}(\mathcal{N})}
\end{split}     \end{equation}
and we refer to
$\smallint_S^T X_r \,\delta W_r$ as the \textbf{classical Skorohod integral}. 
\end{definition}

The following lemma will be applied in the proof of Proposition~\ref{l:Skorohod.on.T}.
\begin{lemma} \label{l:dependence.Malliavin.on.filtration}
Assume Setting~\ref{setting4} and let $s, t \in [S, T]$ satisfy that $s < t$.
Then
\begin{equation} \label{eq:inclusions.mathbbD}
\mathbb{D}^{(1, 2)}( \mathbb P, \mathbb{F}_S, W; \spaceH )
\subseteq
\mathbb{D}^{(1, 2)}( \mathbb P, \mathbb{F}_{[S, s] \cup [t, T]}, W|_{[s, t]\times\Omega}; \spaceH )
\end{equation}
and for all $F \in \mathbb{D}^{(1, 2)}( \mathbb P, \mathbb{F}_S, W; \spaceH )$ it holds $\lambda_{[s,t]}\otimes\mathbb{P}$-a.e.\ that
\begin{equation} \label{eq:inclusions.mathbbD.equality2}
\begin{split} 
\Big( \mathcal D( \mathbb P, \mathbb{F}_S, W; \spaceH ) F \Big) \Big |_{[s, t]\times\Omega}
=
\mathcal D( \mathbb P, \mathbb{F}_{[S, s] \cup [t, T]}, W|_{[s, t]\times\Omega}; \spaceH )F.
\end{split} 
\end{equation}
\end{lemma}

\begin{proof}[Proof of Lemma~\ref{l:dependence.Malliavin.on.filtration}.]
Throughout this proof let 
$F \in \mathcal S( \mathbb P, \mathbb{F}_S, W; \spaceH )$,
let
$n \in \N$, 
$ \phi_1, \dots, \phi_n \in \mathcal L^2( \lambda_{[S, T]}; \spaceU)$, 
$f \in C_b^{ \! \infty, \mathfrak{S}( \mathbb{F}_S \cup \mathcal N)}( \R^n \times \Omega, \R)$,
and
$h \in \spaceH $
satisfy that it holds $ \mathbb P$-a.s.\ that
\begin{equation}
F= f \Big( \smallint_S^T \phi_1(r) \,dW_r , \dots, \smallint_S^T \phi_n (r) \,dW_r \Big) h,
\end{equation}
and let 
$g \in C_b^{ \! \infty, \mathbb{F}_{[S, s] \cup [t, T]}}( \R^n \times \Omega, \R)$
be a function such that for all 
$(x_1, \dots, x_n) \in \R^n$
it holds $ \mathbb P$-a.s.\ that
\begin{equation}
g(x_1, \dots, x_n)
=
f \Big(x_1 + \smallint_S^s \phi_1(r) \,dW_r + \smallint_t^T \phi_1(r) \,dW_r, \dots, x_n + \smallint_S^s \phi_n(r) \,dW_r + \smallint_t^T \phi_n(r) \,dW_r \Big).
\end{equation}
Then it holds $ \mathbb P$-a.s.\ that 
\begin{equation}
F= g \Big( \smallint_s^t \phi_1 (r) \,dW_r , \dots, \smallint_s^t \phi_n (r) \,dW_r \Big) h. 
\end{equation}
This implies that $F\in \mathcal S( \mathbb P, \mathbb{F}_{[S, s] \cup [t, T]}, W|_{[s, t]\times\Omega}; \spaceH )$.
Next for all 
$i \in \{1, \dots, n \}$
it holds $ \mathbb P$-a.s.\ that
\begin{equation}
\frac{ \partial f }{ \partial x_{i}} 
\Big( \smallint_S^T \phi_1(r) \,dW_r , \dots, \smallint_S^T \phi_n(r) \,dW_r \Big) 
=
\frac{ \partial g }{ \partial x_{i}} 
\Big( \smallint_s^t \phi_1 (r) \,dW_r , \dots, \smallint_s^t \phi_n (r) \,dW_r \Big).
\end{equation}
It follows that it holds $ \lambda_{[s,t]}\otimes\mathbb P$-a.e.\ that 
\begin{equation} \label{2681}
\begin{split}
&\Big( \mathcal D( \mathbb P, \mathbb{F}_S, W; \spaceH ) F \Big) \Big |_{[s, t]\times\Omega} 
\\&
=
\sum_{i=1}^n \frac{ \partial f }{ \partial x_{i}} 
\Big( \smallint_S^T \phi_1(r) \,dW_r , \dots, \smallint_S^T \phi_n(r) \,dW_r \Big) 
\big( \phi_i |_{[s, t]} \big) 
h
=
\sum_{i=1}^n \frac{ \partial g }{ \partial x_{i}} 
\Big( \smallint_s^t \phi_1 (r) \,dW_r , \dots, \smallint_s^t \phi_n (r) \,dW_r \Big) 
\big( \phi_i |_{[s, t]} \big) 
h
\\&=
\mathcal D( \mathbb P, \mathbb{F}_{[S, s] \cup [t, T]}, W|_{[s, t]\times\Omega}; \spaceH )F.
\end{split}
\end{equation}
Equation~\eqref{2681} implies that 
\begin{equation} \label{blutwurst}
\begin{split}
&\| F \|_{ \mathbb{D}^{(1, 2)}( \mathbb P, \mathbb{F}_{[S, s] \cup [t, T]}, W|_{[s, t]\times\Omega}; \spaceH )}^2
\\&=
\mathbb E \Big[ \| F \|^2_{\spaceH} + \| \mathcal D( \mathbb P, \mathbb{F}_{[S, s] \cup [t, T]}, W|_{[s, t]\times\Omega}; \spaceH )F \|^2_{L^2( \lambda_{[s, t]}; \spaceH )} \Big] 
\\&=
\mathbb E \Big[ \| F \|^2_{\spaceH} + \Big\| \Big( \mathcal D( \mathbb P, \mathbb{F}_S, W; \spaceH ) F \Big) \Big |_{[s, t]\times\Omega} \Big\|^2_{L^2( \lambda_{[s, t]}; \spaceH )} \Big] 
\\&\leq
\mathbb E \big[ \| F \|^2_{\spaceH} + \| \mathcal D( \mathbb P, \mathbb{F}_S, W; \spaceH )F \|^2_{L^2( \lambda_{[S, T]}; \spaceH )} \big]
=
\| F \|_{ \mathbb{D}^{(1, 2)}( \mathbb P, \mathbb{F}_S, W; \spaceH )}^2.
\end{split}
\end{equation}
Since 
$F \in \mathcal S( \mathbb P, \mathbb{F}_S, W; \spaceH )$
was chosen arbitrarily it follows that
\begin{equation}
\operatorname{span}( \mathcal S( \mathbb P, \mathbb{F}_S, W; \spaceH ))
\subseteq
\operatorname{span}( \mathcal S( \mathbb P, \mathbb{F}_{[S, s] \cup [t, T]}, W|_{[s, t]\times\Omega}; \spaceH )).
\end{equation}
This and inequality~\eqref{blutwurst} yield the inclusion~\eqref{eq:inclusions.mathbbD}, and
equation~\eqref{2681} implies
equation~\eqref{eq:inclusions.mathbbD.equality2}.
The proof of Lemma~\ref{l:dependence.Malliavin.on.filtration} is thus completed.
\end{proof}

The following result, Proposition~\ref{l:Skorohod.on.T}, shows how to change the domain of integration for Skorohod integrals. 
\begin{prop} \label{l:Skorohod.on.T}
Assume Setting~\ref{setting4},
let
$X \in L^0( \mathbb P; L^2( \lambda_{[S, T]}; \R^{d\times m}))$,
and
let $s, t \in [S, T]$ satisfy that $s < t$.
Then the following two statements are equivalent:
\begin{enumerate}[(i)]
  \item  
  It holds that
$X|_{[s, t]\times\Omega}$ is 
$( \mathbb P, \mathbb{F}_{[S,s]\cup[t,T]}, W|_{[s, t]\times\Omega}; \spaceH )$-Skorohod-integrable.
\item
It holds that
$\1_{[s, t]} X$ is $( \mathbb P, \mathbb{F}_S, W; \spaceH )$-Skorohod-integrable.
\end{enumerate}
If any of these two statements is true, then it holds
$ \mathbb P$-a.s.\ that 
\begin{equation} \label{geor}
\smallint_s^t X_r \,\delta W_r^{\mathbb{F}_{[S,s]\cup[t,T]}}
= 
\smallint_S^T \1_{[s, t]} (r) X _r \,\delta W_r^{\mathbb{F}_{S}}.
\end{equation}
\end{prop}

\begin{proof}[Proof of Proposition~\ref{l:Skorohod.on.T}]
\textbf{`(i) implies (ii)':}
Assume that the process $X|_{[s, t]\times\Omega}$ is 
$( \mathbb P, \mathbb{F}_{[S,s]\cup[t,T]}, W|_{[s, t]\times\Omega}; \spaceH )$-Skorohod-integrable. 
This implies that 
$ \1_{[s, t]} X \in L^2( \mathbb P|_{\mathbb{F}_T}; L^2( \lambda_{[S, T]}; \R^{d\times m}))$.
Lemma~\ref{l:dependence.Malliavin.on.filtration}, the definition of the Skorohod integral,
and the Cauchy-Schwarz inequality imply for all 
$F \in \mathbb D^{(1, 2)}( \mathbb P, \mathbb{F}_S, W; \spaceH )$
that
\begin{equation} \label{locher5}
\begin{split}
&\mathbb E \Big[ \Big \langle \mathcal D( \mathbb P, \mathbb{F}_{S}, W; \spaceH ) F, \1_{[s, t]} X
\Big \rangle_{ L^2( \lambda_{[S, T]}; \R^{d\times m} )} \Big]
\\& =
\mathbb E \Big[ \Big \langle ( \mathcal D( \mathbb P, \mathbb{F}_{S}, W; \spaceH ) F) |_{[s, t]\times\Omega} , 
X |_{[s, t]\times\Omega} \Big \rangle_{L^2( \lambda_{[s, t]}; \R^{d\times m} )}
\Big]
\\&
=
\mathbb E \Big[
\Big \langle \mathcal D( \mathbb P, \mathbb{F}_{[S, s] \cup [t, T]}, W|_{[s, t]\times\Omega}; \spaceH ) F, 
X |_{[s, t]\times\Omega}
\Big \rangle_{L^2( \lambda_{[s, t]}; \R^{d\times m})}
\Big]
\\& =
\mathbb E \Big[ \Big \langle F, 
\smallint_s^t X_r \,\delta W_r^{\mathbb{F}_{[S,s]\cup[t,T]}}
\Big \rangle_{\spaceH} \Big]
\\&
\leq
\Big \|
\smallint_s^t X_r \,\delta W_r^{\mathbb{F}_{[S,s]\cup[t,T]}}
\Big \|_{L^2( \mathbb P; \spaceH)}
\cdot
\| F \|_{L^2( \mathbb P; \spaceH)}
< \infty.
\end{split}
\end{equation}
We conclude that $\1_{[s, t]} X$ is $( \mathbb P, \mathbb{F}_S, W; \spaceH)$-Skorohod-integrable.

\textbf{`(ii) implies (i)':}
Assume that $\1_{[s, t]} X$ is $( \mathbb P, \mathbb{F}_S, W; \spaceH)$-Skorohod-integrable.
This implies that it holds that $X|_{[s, t]\times\Omega} \in L^2( \mathbb P|_{ \mathbb{F}_T}; L^2( \lambda_{[s, t]}; \R^{d\times m}))$.
Lemma~\ref{l:dependence.Malliavin.on.filtration} and the definition of the Skorohod integral yield for all
$F \in \mathbb D^{(1, 2)}( \mathbb P, \mathbb{F}_S, W; \spaceH)$
that
$F \in \mathbb{D}^{(1, 2)}( \mathbb P, \mathbb{F}_{[S, s] \cup [t, T]}, W|_{[s, t]\times\Omega}; \spaceH)$
and that
\begin{equation} \label{locher6}
\begin{split}
&\mathbb E \Big[
\Big \langle 
\big( \mathcal D( \mathbb P, \mathbb{F}_{[S, s] \cup [t, T]}, W|_{[s,t]\times\Omega}; \spaceH) F \big) , 
X \big |_{[s, t]\times\Omega}
\Big \rangle_{L^2( \lambda_{[s, t]}; \R^{d\times m})}
\Big]
\\&=\mathbb E \Big[
\Big \langle 
\big( \mathcal D( \mathbb P, \mathbb{F}_{S}, W; \spaceH) F \big) \big |_{[s, t]\times\Omega}, 
X \big |_{[s, t]\times\Omega}
\Big \rangle_{L^2( \lambda_{[s, t]}; \R^{d\times m})}
\Big]
\\
& =
\mathbb E \big[ \big \langle \mathcal D( \mathbb P, \mathbb{F}_{S}, W; \spaceH ) F, \1_{[s, t]} X 
\big \rangle_{ L^2( \lambda_{[S, T]}; \R^{d\times m})} \big]
\\& 
=
\mathbb E \big[ \big \langle F,
\smallint_S^T \1_{[s, t]} (r) X _r \,\delta W_r^{\mathbb{F}_{S}}
 \big \rangle_{\spaceH} \big]
\\&
\leq
\Big \|
\smallint_S^T \1_{[s, t]} (r) X _r \,\delta W_r^{\mathbb{F}_{S}}
\Big \|_{L^2( \mathbb P; \spaceH)}
\cdot
\| F \|_{L^2( \mathbb P; \spaceH)}
< \infty.
\end{split}
\end{equation}
Lemma~\ref{l:suff.rich} shows that
$ \operatorname{span}( \mathcal S( \mathbb P, \mathbb{F}_S, W; \spaceH))$ 
is dense in 
$L^2( \mathbb P|_{ \mathbb{F}_T}; \spaceH)$.
This, \eqref{locher5},~\eqref{locher6}, and the definition of the Skorohod
integral
imply that
$X|_{[s, t]\times\Omega}$ is 
$( \mathbb P, \mathbb{F}_{[S,s]\cup[t,T]}, W|_{[s, t]\times\Omega}; \spaceH)$-Skorohod-integrable
and
that it holds 
$ \mathbb P$-a.s.\ that 
\begin{equation}
\smallint_s^t X_r \,\delta W_r^{\mathbb{F}_{[S,s]\cup[t,T]}}
= 
\smallint_S^T \1_{[s, t]} (r) X _r \,\delta W_r^{\mathbb{F}_{S}}.
\end{equation}
The proof of Proposition~\ref{l:Skorohod.on.T} is thus completed.
\end{proof}

It is well-known (e.g., Nualart~\cite[Proposition 1.3.11]{Nualart2006})
that the classical Skorohod integral generalizes the It\^o integral restricted to square-integrable integrands which are adapted to the Brownian filtration. 
The following result, Proposition~\ref{l:Skorohod.generalizes.Ito}, generalizes this.
The proof of Lemma~\ref{l:Skorohod.generalizes.Ito} is analogous to the proof of
Nualart~\cite[Proposition 1.3.11]{Nualart2006}
and is therefore omitted.

\begin{prop} \label{l:Skorohod.generalizes.Ito}
Assume Setting~\ref{setting4},
let $s, t \in [S, T]$ satisfy $s<t$, let
$ \tilde{\mathbb{F}} = ( \tilde{\mathbb{F}}_r)_{r \in [s, t]}$
be a filtration with the property that for all $r \in [s, t]$ it holds that 
$
\tilde{\mathbb{F}}_r 
=
\mathfrak{S}(
\mathfrak{S}(W_u - W_s \colon u \in [s, r])
\cup\mathbb{F}_{[S,s]\cup[t,T]}
)
$
and let
$X \in \mathcal{L}^2( \mathbb P; \mathcal{L}^2( \lambda_{[s, t]}; \R^{d\times m}))$
be $ \tilde{\mathbb{F}}$-predictable.
Then $X$ 
is 
$( \mathbb P, \mathbb{F}_{[S,s]\cup[t,T]}, W|_{[s, t]\times\Omega}; \spaceH)$-Skorohod-integrable 
and it holds $ \mathbb P$-a.s.\ that 
\begin{equation} \label{eq:Skorohod.generalizes.Ito}
\smallint_s^t X_r \,\delta W_r^{\mathbb{F}_{[S,s]\cup[t,T]}}
= 
\smallint_s^t X_r \,dW_r.
\end{equation}
\end{prop}

The next result, Lemma~\ref{l:Skorohod.integral.closed}, proves that if a sequence of integrals converges weakly and has uniformly bounded Skorohod integrals, then the limit is Skorohod-integrable and the sequence of Skorohod integrals of the sequence converges weakly. 
Lemma~\ref{l:Skorohod.integral.closed}
follows immediately from the definition of the Skorohod integral
and its proof is therefore omitted.
\begin{lemma} \label{l:Skorohod.integral.closed}
Assume Setting~\ref{setting4}, let
$X \in L^2( \mathbb P|_{ \mathbb{F}_T }; L^2( \lambda_{[S, T]}; \R^{d\times m}))$,
and let 
$(X_n)_{n \in \N }\subseteq\operatorname{Dom}_ \delta( \mathbb P, \mathbb{F}_S, W; \spaceH)$ be a sequence which satisfies that
$ \sup_{n \in \N} \| \delta ( \mathbb P, \mathbb{F}_S, W; \spaceH) (X_n) \|_{L^2 ( \mathbb P|_{ \mathbb{F}_T }; \spaceH)} < \infty$
and
which converges to $X$ in the weak topology of 
$L^2( \mathbb P|_{ \mathbb{F}_T }; L^2( \lambda_{[S, T]}; \R^{d\times m}))$.
Then 
$X \in \operatorname{Dom}_ \delta( \mathbb P, \mathbb{F}_S, W; \spaceH)$
and 
$( \delta ( \mathbb P, \mathbb{F}_S, W; \spaceH) (X_n))_{n \in \N}$
converges to 
$ \delta ( \mathbb P, \mathbb{F}_S, W; \spaceH) (X)$
in the weak topology of 
$L^2( \mathbb P|_{ \mathbb{F}_T }; \spaceH)$.
\end{lemma}

\section*{Acknowledgements}
This project has been partially supported 
by the Deutsche For\-schungs\-ge\-mein\-schaft (DFG, German Research Foundation) via 
RTG 2131 {\it High-dimensional Phenomena in Probability -- Fluctuations and Discontinuity}.
The third author has been partially supported by the startup fund project of Shenzhen Research Institute of Big Data under grant No.\ T00120220001. 
The third author also gratefully acknowledges the Cluster of Excellence EXC 2044-390685587, Mathematics M\"unster: Dynamics-Geometry-Structure funded by the Deutsche Forschungsgemeinschaft (DFG, German Research Foundation).



\def\cprime{$'$} \def\cprime{$'$}

\end{document}